\documentclass[12pt,leqno]{article}
\usepackage{amsfonts,amsthm,amsmath}

\usepackage{lscape}

\newtheorem{thm}{Theorem}[section]
\newtheorem{prop}[thm]{Proposition}
\newtheorem{lem}[thm]{Lemma}
\newtheorem{cor}[thm]{Corollary}
\theoremstyle{remark}
\newtheorem{rem}[thm]{Remark}

\newcommand{\ZZ}{\mathbb{Z}}
\newcommand{\RR}{\mathbb{R}}

\DeclareMathOperator{\Aut}{Aut}

\begin{document}
%%%%%%%%%%%%%%%%%%%%%%%%%%%%%%%%%%

\title{On the Existence of Frames of Some Extremal 
Odd Unimodular
Lattices and Self-Dual $\ZZ_{k}$-Codes}

\author{
Masaaki Harada\thanks{
Department of Mathematical Sciences,
Yamagata University,
Yamagata 990--8560, Japan. 
email: mharada@sci.kj.yamagata-u.ac.jp}
and 
Tsuyoshi Miezaki\thanks{
Faculty of Education, Art and Science, 
Yamagata University, 
Yamagata 990--8560, Japan. 
email: miezaki@e.yamagata-u.ac.jp
}
}

\maketitle

\begin{abstract}
For some extremal (optimal) odd unimodular lattices $L$
in dimensions $n=12,16,20,32,36,40$ and $44$,
we determine all positive integers $k$ such that $L$ contains a
$k$-frame.
This result yields the existence of 
an extremal Type~I $\ZZ_{k}$-code of lengths 
$12,16,20,32,36,40$ and $44$ and a near-extremal Type~I $\ZZ_k$-code of
length $28$ for positive integers $k$ with only a few exceptions.
\end{abstract}

%%%%%%%%%%%%%%%%%%%%%%%%%%%%%%%%%%
\section{Introduction}\label{Sec:1}

Self-dual codes and unimodular lattices are studied  from several
viewpoints (see~\cite{SPLAG} for an extensive bibliography). Many
relationships between self-dual codes and unimodular lattices are
known and there are similar situations between two subjects.
%% As a remarkable example
As a typical example,
it is known that a unimodular lattice $L$ contains a $k$-frame 
if and only if there exists a self-dual $\ZZ_{k}$-code $C$ such that
$L$ is isomorphic to the lattice obtained from $C$ by
Construction A,
where $\ZZ_{k}$ is the ring of integers modulo $k$.

As described in~\cite{RS-Handbook},
self-dual codes are an important class of linear codes for both
theoretical and practical reasons.
It is a fundamental problem to classify self-dual codes
of modest lengths 
and determine the largest minimum weight among self-dual codes
of that length.
Type~II $\ZZ_{2k}$-codes were defined in~\cite{BDHO} 
as a class of self-dual codes,
%with the property that all Euclidean weights are divisible by $4k$,
which are related to even unimodular lattices.
% Much work has been done concerning 
% Type~II $\ZZ_{2k}$-codes.
For binary Type~II codes,
much work has been done concerning the above fundamental problem
(see e.g.~\cite{BHM, CPS, SPLAG, RS-Handbook}).
%%%%%%%%%%%
For general $k$,
if $C$ is a Type~II $\ZZ_{2k}$-code of length $n \le 136$ 
then we have the bound on the minimum Euclidean weight
$d_E(C)$ of $C$ as follows:
$d_E(C) \le 4k \left\lfloor \frac{n}{24} \right\rfloor +4k$
for every positive integer $k$ (see \cite{HM12}).
We say that a Type~II $\ZZ_{2k}$-code meeting the bound 
with equality is {\em extremal} for length $n \le 136$.
It was shown in~\cite{Chapman, GH01} that
the Leech lattice, which is one of the most remarkable lattices,
contains a $2k$-frame for every positive integer $k \ge 2$.
This result yields the existence of an extremal Type~II $\ZZ_{2k}$-code of
length $24$  for every positive integer $k$.
Recently, the existence of
an extremal Type~II $\ZZ_{2k}$-code of length $n=32,40,48,56,64$ 
has been established by the authors \cite{HM12}
for every positive integer $k$.
This was done by finding a $2k$-frame in some
extremal even unimodular lattices in these dimensions $n$.

%%%%%%%%%%%%%%%%%%%%%
% The existence of a $2k$-frame ($k \ge 2$) in the Leech lattice 
% $\Lambda$ is
% used in \cite{Shimakua} to show that 
% the Moonshine vertex operator algebra $V^\natural$
% contains a vertex operator subalgebra 
% isomorphic to the tensor product of $24$ copies of
% the fixed-point subspace $V^+_L$ of the lattice
% vertex operator algebra $V_L$, where
% $L=\ZZ \alpha$ with $(\alpha,\alpha)=2k$
% and $\alpha \in \Lambda$.
% It would be worthwhile to establish the existence of a 
% $k$-frame in some extremal (odd) unimodular lattices in this
% direction.
% % There is not more work concerning Type~I codes than Type~II codes.
% % This is a motivation of our study of extremal Type~I $\ZZ_k$-codes for
% % general $k$.
% % On the other hand, the odd Leech lattice contains a $k$-frame for 
% % every positive integer $k$ with $k \ge 3$~\cite{Miezaki}.
% % This also motivates our investigation of the existence of
% % a $k$-frame of extremal odd unimodular lattices.
% Recently, it was shown in~\cite{Miezaki} that the odd Leech 
% lattice contains a $k$-frame for every positive integer $k$ with 
% $k \ge 3$.
% This also motivates our investigation of the existence of
% a $k$-frame in extremal odd unimodular lattices.
%%%%%%%%%%%%%%%%%%%%%
Recently, it was shown in~\cite{Miezaki} that the odd Leech 
lattice contains a $k$-frame for every positive integer $k$ with 
$k \ge 3$.
This motivates our investigation of the existence of
a $k$-frame in extremal odd unimodular lattices.
% The main aim of this paper is to establish the existence of
% a $k$-frame in some extremal (optimal) odd unimodular lattices 
% in dimensions $n=12,16,20,32,36,40$ and $44$,
% for positive integers $k$ with only a few small exceptions.
In this paper, 
for some extremal (optimal) odd unimodular lattices $L$
in dimensions $n=12,16,20,32,36,40$ and $44$,
we determine all integers $k$ such that $L$ contains a
$k$-frame.
This result yields the existence of 
an extremal Type~I $\ZZ_{k}$-code of lengths 
$12,16,20,32,36,40$ and $44$ and a near-extremal Type~I $\ZZ_k$-code of
length $28$ for positive integers $k$ with only a few small exceptions.
% The powerful tool in the study of this paper
% is to consider the existence
% of $k$-frames in some extremal odd unimodular lattices.
% Our approach in this paper
% is similar to that in  \cite{HM12}
% (see also \cite{Chapman, GH01, Miezaki}).

%%%%%%%%%%%%%%%%%%%%%%%%%%%%%
This paper is organized as follows. In Section~\ref{sec:Pre}, we
give definitions and some basic properties of self-dual codes
and unimodular lattices used in this paper.
The notion of extremal Type~I $\ZZ_k$-codes of length $n$
is given for $n \le 48$ and $k \ge 2$.
Lemma \ref{lem:frame} gives a reason why we consider 
unimodular lattices in only dimension $n\equiv 0 \pmod 4$.
%% Section 3
In Section \ref{sec:NT},
using the theory of modular forms (see~\cite{Miyake} for details),
we derive some number theoretical results (Theorem~\ref{thm:prime}),
which are used in Section \ref{sec:frame}.
%% Section 4
In Section~\ref{sec:frame}, we provide a method for
constructing $m$-frames in unimodular lattices, which
are constructed from some self-dual $\ZZ_k$-codes by Construction A
(Proposition \ref{prop:const}).
This method is a generalization of Propositions 3.3 and 3.6 in \cite{HM12}.
% Some special cases of this method can be found 
% in \cite{Chapman,HM12,Miezaki}.
Using Theorem~\ref{thm:prime} and Proposition \ref{prop:const},
we give $k$-frames in 
the unique extremal odd unimodular lattice
in dimensions $n=12,16$ and 
some extremal (optimal) odd unimodular lattices $L$
in dimensions $n=20,32,36,40$ and $44$,
which are listed in Table \ref{Tab:L}, 
for all positive integers $k$ satisfying the condition
$(\star)$ in Table \ref{Tab:L}
(Lemma \ref{lem:key}).
%% Section 5
In Section \ref{sec:main}, 
some extremal (near-extremal) Type~I $\ZZ_{k}$-codes are constructed
for some integers $k$.  %, especially the above exceptions $k$.
Then we establish the existence of
a $k$-frame in the extremal (optimal) unimodular lattices $L$
in dimension $n=12,16,20,28, 32,36$, which are listed in Table \ref{Tab:L} 
(except only lattices 
$A_3(C_{20,3}(D'_{10}))$ and $A_5(C_{20,5}(D''_{10}))$),
for every positive integer $k$ with $k \ge \min(L)$, where $\min(L)$
denotes the minimum norm of $L$.
%(Theorems \ref{thm:D12}, \ref{thm:D82}, \ref{thm:20}, \ref{thm:28}, 
%\ref{thm:32}, \ref{thm:36}).
%%%
As a consequence, we show that
the $32$-dimensional Barnes--Wall lattice $BW_{32}$
contains a $2k$-frame
for every positive integer $k$ with $k \ge 2$. 
%%%%
We also discuss the positivity of coefficients of 
the theta series of some extremal (optimal) unimodular lattices
in dimension $n \le 36$.
%%%
When $n=40,44$, it is shown that there is an extremal odd unimodular 
lattice in dimension $n$ containing a $k$-frame
for every positive integer $k$ with $k \ge 4$. 
%(Lemmas \ref{lem:40}, \ref{lem:44}).
As a consequence,  the existence of 
an extremal Type~I $\ZZ_{k}$-code of lengths 
$n=12,16,20,32,36,40, 44$ 
and a near-extremal Type~I $\ZZ_k$-code of
length $n=28$ is established for a positive integer $k$,
where 
$k \ne 1,3$ if $n =32$ and $k \ne 1$ otherwise.
%(Corollaries \ref{cor:12}, \ref{cor:16}, \ref{cor:20},
%\ref{cor:28}, \ref{cor:32}, \ref{cor:36},
%Theorems \ref{thm:40}, \ref{thm:40}).
Finally, in Section 6, we investigate the existence of
a $k$-frame in optimal odd unimodular lattices in dimension $48$.

All computer calculations in this paper were
done by {\sc Magma}~\cite{Magma}.

%{\bf 
%\begin{itemize}
%
%\item
%Will we consider the
%positivity of coefficients of theta series of extremal odd unimodular
%lattices?
%
%\end{itemize}
%}

%%%%%%%%%%%%%%%%%%%%%%%%%%%%%%%%%%
\section{Preliminaries}\label{sec:Pre}

In this section, we give definitions and some basic properties 
of self-dual codes and unimodular lattices 
used in this paper.

%%%%
\subsection{Self-dual codes}
Let $\ZZ_{k}$ be the ring 
of integers modulo $k$, where $k$ 
is a positive integer. 
In this paper, we always assume that $k\geq 2$ and 
we take the set $\ZZ_{k}$ to be 
$\{0,1,\ldots,k-1\}$.
%%using whichever form is more convenient.  
A $\ZZ_{k}$-code $C$ of length $n$
(or a code $C$ of length $n$ over $\ZZ_{k}$)
is a $\ZZ_{k}$-submodule of $\ZZ_{k}^n$.
A $\ZZ_2$-code and a $\ZZ_3$-code
are called binary and ternary, respectively.
% The Euclidean weight of a codeword 
% $x=(x_1,x_2,\ldots,x_n)$ is 
% $\sum_{i=1}^n \min\{x_i^2,(k-x_i)^2\}$.
The Euclidean weight of a codeword $x=(x_1,\ldots,x_n)$ of $C$ is
$\sum_{\alpha=1}^{\lfloor k/2 \rfloor}n_\alpha(x) \alpha^2$, 
where $n_{\alpha}(x)$ denotes
the number of components $i$ with $x_i \equiv \pm \alpha \pmod k$ 
$(\alpha=1,2,\ldots,\lfloor k/2 \rfloor)$.
It is trivial that the Euclidean weight is the same
as the (usual) Hamming weight for the case $k=2,3$.
The minimum Euclidean weight $d_E(C)$ of $C$ is the smallest Euclidean
weight among all nonzero codewords of $C$.

%%%%%
A $\ZZ_{k}$-code $C$ is {\em self-dual} if $C=C^\perp$, where
the dual code $C^\perp$ of $C$ is defined as 
$C^\perp = \{ x \in \ZZ_{k}^n \mid x \cdot y = 0$ for all $y \in C\}$
under the standard inner product $x \cdot y$. 
A {\em Type~II} $\ZZ_{2k}$-code was defined in
\cite{BDHO} as a self-dual code with the property that all
Euclidean weights are divisible by $4k$.
It is known that a Type~II $\ZZ_{k}$-code of length $n$ exists
if and only if $n$ is divisible by eight and $k$ is even \cite{BDHO}.
A self-dual code which is not Type~II is called {\em Type~I}.

Two self-dual $\ZZ_{k}$-codes $C$ and $C'$ are {\em equivalent} 
if there exists a monomial $(\pm 1, 0)$-matrix $P$ with 
$C' = C \cdot P = \{ x P\:|\: x \in C\}$.  
The automorphism group $\Aut(C)$ of $C$ is the group of all
monomial $(\pm 1, 0)$-matrices $P$ with
$C = C \cdot P$.

%%%%
\subsection{Unimodular lattices}\label{sec:2U}
A (Euclidean) lattice $L \subset \RR^n$
in dimension $n$
is {\em unimodular} if
$L = L^{*}$, where
the dual lattice $L^{*}$ of $L$ is defined as
$\{ x \in {\RR}^n \mid (x,y) \in \ZZ \text{ for all }
y \in L\}$ under the standard inner product $(x,y)$.
Two lattices $L$ and $L'$ are {\em isomorphic}, denoted $L \cong L'$,
if there exists an orthogonal matrix $A$ with
$L' = L \cdot A$.
The norm of a vector $x$ is defined as $(x, x)$.
The minimum norm $\min(L)$ of a unimodular
lattice $L$ is the smallest norm among all nonzero vectors of $L$.
%%%%%
The theta series $\theta_{L}(q)$ of $L$ is the formal power
series $\theta_{L}(q) = \sum_{x \in L} q^{(x,x)}$.
The kissing number of $L$ is the second nonzero coefficient of the
theta series.
%%%%%
A unimodular lattice with even norms is said to be {\em even}, 
and that containing a vector of odd norm is said to be {\em odd}.
% An odd  unimodular lattice exists for every dimension.
%Indeed, $\ZZ^n$ is an odd  unimodular lattice in dimension $n$.
An even unimodular lattice in dimension $n$
exists if and only
if $n \equiv 0 \pmod 8$, while an odd unimodular lattice
exists for every dimension. 
Two lattices $L$ and $L'$ are {\em neighbors} if
both lattices contain a sublattice of index $2$
in common.

%%% Definition of neighbor and shadow
Let $L$ be a unimodular lattice.
Define $L_0=\{x \in L \mid (x,x) \equiv 0 \pmod 2\}$.
Then $L_0$ is a sublattice of $L$ of index $2$ if $L$ is
odd and $L_0=L$ if $L$ is even.
The shadow $S$ of $L$ is defined as
$S=L_0^* \setminus L$ if  $L$ is odd 
and as $S=L$ if $L$ is even \cite{CS-odd}.
%%%%
Now suppose that $L$ is an odd unimodular lattice.
Then there are cosets $L_1,L_2,L_3$ of $L_0$ such that
$L_0^* = L_0 \cup L_1 \cup L_2 \cup L_3$, where
$L = L_0  \cup L_2$ and $S = L_1 \cup L_3$.
If $L$ is an odd unimodular lattice in dimension divisible by eight,
then $L$ has two even unimodular neighbors
of $L$, namely, $L_0 \cup L_1$ and $L_0 \cup L_3$.

Rains and Sloane \cite{RS-bound} showed
that a unimodular lattice $L$ in dimension $n$
has minimum norm $\min(L) \le 2 \lfloor \frac{n}{24} \rfloor+2$
unless $n=23$ when $\min(L) \le 3$ (see~\cite{Siegel}
for the case that $L$ is even).
A unimodular lattice meeting the bound
with equality is called {\em extremal}.
Gaulter~\cite{Gaulter} showed that
any unimodular lattice in dimension $24k$
meeting the upper
bound has to be even, which was conjectured in \cite{RS-bound}.
Hence,  an odd unimodular lattice $L$
in dimension $24k$
satisfies $\min(L) \le 2k+1$.
We say that an odd unimodular lattice with the largest minimum norm
among all odd unimodular lattices in that dimension is {\em optimal}.

% Let $L$ be an odd unimodular lattice and let $L_0$ denote its
% sublattice of vectors of even norms.
% Then $L_0$ is a sublattice of $L$ of index $2$.
% $L_0^* \setminus L$ is called the shadow $S$ of $L$~\cite{CS-odd}.
% There are cosets $L_1,L_2,L_3$ of $L_0$ such that
% $L_0^* = L_0 \cup L_1 \cup L_2 \cup L_3$, where
% $L = L_0  \cup L_2$ and $S = L_1 \cup L_3$.
% If $L$ is an odd unimodular lattice in dimension divisible by eight,
% then $L$ has two even unimodular neighbors
% of $L$, namely, $L_0 \cup L_1$ and $L_0 \cup L_3$.

%%%%
\subsection{Construction A and $k$-frames}

We give a method to construct 
unimodular lattices from self-dual $\ZZ_{k}$-codes, which 
is referred to as {\em Construction A} (see~\cite{{BDHO},{HMV}}). 
Let $\rho$ be a map from $\ZZ_{k}$ to $\ZZ$ sending $0, 1, \ldots , k-1$ 
to $0, 1, \ldots , k-1$, respectively.
If $C$ is a  self-dual $\ZZ_{k}$-code of length $n$, then 
the lattice 
\[
A_{k}(C)=\frac{1}{\sqrt{k}}\{\rho (C) +k \ZZ^{n}\} 
\]
is a unimodular lattice in dimension $n$, where 
$
\rho (C)=\{(\rho (c_{1}), \ldots , \rho (c_{n})) 
\mid (c_{1}, \ldots , c_{n}) \in C\}. 
$
The minimum norm of $A_{k}(C)$ is $\min\{k, d_{E}(C)/k\}$.
Moreover, $C$ is a Type~II $\ZZ_{2k}$-code if and only if
$A_{2k}(C)$ is an even unimodular lattice~\cite{BDHO}.

%%%%%
A set $\{f_1, \ldots, f_{n}\}$ of $n$ vectors $f_1, \ldots, f_{n}$ of 
a unimodular lattice $L$ in dimension $n$ with
$ ( f_i, f_j ) = k \delta_{i,j}$
is called a {\em $k$-frame} of $L$,
where $\delta_{i,j}$ is the Kronecker delta.
It is known that a unimodular lattice $L$ contains a $k$-frame 
if and only if there exists a self-dual $\ZZ_{k}$-code $C$ with 
$A_{k}(C) \cong L$ (see \cite{HMV}).
%% In addition, an even unimodular lattice $L$ contains a $2k$-frame 
%% if and only if there exists a Type~II $\ZZ_{2k}$-code $C$ 
%% with $A_{2k}(C) \cong L$ (see~\cite{{Chapman},{HMV}}). 

%% The following lemma is useful to provide the
%% existence of $k$-frames.
%% More precisely, it is enough to consider 
%% a $p$-frame of an odd  unimodular lattice
%% for each prime $p$.
By the following lemma, it is enough to consider 
a $p$-frame in an odd  unimodular lattice
for each prime $p$.
The lemma also gives a reason why we consider 
unimodular lattices in only dimension $n\equiv 0 \pmod 4$.

\begin{lem}[{Chapman~\cite[Lemma 5.1]{Chapman}}]
\label{lem:frame}
% Let $n$ be a positive integer divisible by four.
If a lattice $L$ in dimension $n\equiv 0 \pmod 4$ contains a $k$-frame, then
$L$ contains a $km$-frame for every positive integer $m$.
\end{lem}

%%%%
\subsection{Upper bounds on the minimum Euclidean weights}

A self-dual $\ZZ_k$-code $C$ of length $n$ satisfies the bound:
\begin{equation}\label{eq:kbound}
d_E(C) \le
\begin{cases}
4 \lfloor \frac{n}{24} \rfloor+4 & \text{ if } 
                      k=2, n \not\equiv 22 \pmod{24} \\
4 \lfloor \frac{n}{24} \rfloor+6 & \text{ if } 
                      k=2, n \equiv 22 \pmod{24} \\
%%%%%
3 \lfloor \frac{n}{12} \rfloor +3 & \text{ if } k=3 \\
%%%%%
8 \lfloor \frac{n}{24} \rfloor+8 & \text{ if } k=4, 
                                  n \not\equiv 23 \pmod{24} \\
8 \lfloor \frac{n}{24} \rfloor+12 & \text{ if } k=4, 
                                  n \equiv 23 \pmod{24},
\end{cases}
\end{equation}
\cite{MPS,Rains,RS-bound}.
Note that 
a binary self-dual code of length divisible by $24$ meeting the bound
must be Type~II \cite{Rains}.

% A binary self-dual code $C$ of length $n$ satisfies the bound
% $d_E(C) \le 4 \lfloor \frac{n}{24} \rfloor+4$,
% except when $n \equiv 22 \pmod{24}$, when the bound is
% $d_E(C) \le 4 \lfloor \frac{n}{24} \rfloor+6$ \cite{Rains}.
% A binary self-dual code of length divisible by $24$ meeting the bound
% must be Type~II \cite{Rains}.
% A ternary self-dual code $C$ of length $n$ satisfies the bound
% $d_E(C) \le 3 \lfloor \frac{n}{12} \rfloor +3$ \cite{MPS}.
% A self-dual $\ZZ_4$-code $C$ of length $n$ satisfies the bound
% $d_E(C) \le 8 \lfloor \frac{n}{24} \rfloor+8$,
% except when $n \equiv 23 \pmod{24}$, when the bound is
% $d_E(C) \le 8 \lfloor \frac{n}{24} \rfloor+12
% %% $ \cite[Theorem 35]{RS-Handbook}.
% $ \cite{RS-bound}.

Although the following lemmas are somewhat trivial, 
we give proofs for the sake of completeness.

\begin{lem}\label{lem:bound}
Let $C$ be a self-dual $\ZZ_k$-code of length $n$.
If $n \ne 23$ and 
$k \ge 2 \lfloor \frac{n}{24} \rfloor+3$,
then $d_E(C) \le 2k \lfloor \frac{n}{24} \rfloor+2k$.
If $n = 23$ and $k \ge 4$ then $d_E(C) \le 3k$.
\end{lem}
\begin{proof}
Since both cases are similar, 
we only consider the case that $n \ne 23$ and 
$k \ge 2 \lfloor \frac{n}{24} \rfloor+3$.
Note that the Euclidean weight of a codeword of $C$ is
divisible by $k$.
Suppose that 
$d_E(C) \ge 2k \lfloor \frac{n}{24} \rfloor+3k$.
Since $\min(A_{k}(C))=\min\{k, d_{E}(C)/k\}$,
$\min(A_{k}(C)) \ge  2 \lfloor \frac{n}{24} \rfloor+3$,
which is a contradiction to the upper bound
on the minimum norms of unimodular lattices.
%% Consider the case that $n=23$ and $k \ge 4$.
%% Suppose that $d_E(C) \ge 4k$.
%% Then $\min(A_{k}(C)) \ge \min\{k,4\} \ge 4$,
%% which is a contradiction.
\end{proof}

\begin{lem}
Let $C$ be a self-dual $\ZZ_k$-code of length $48$.
Then $d_E(C) \le 6k$
for every positive integer $k$ with $k \ge 2$.
\end{lem}
\begin{proof}
By the bound (\ref{eq:kbound}) and Lemma \ref{lem:bound},
it is sufficient to consider the cases for only $k=5,6$.
Assume that $k=5,6$ and $d_E(C) \ge 7k$.
Since $k < d_E(C)/k$, 
$\min(A_k(C))= k$ and the kissing number of $A_k(C)$ is $96$.
Note that unimodular lattices $L$ with $\min(L)=6$ and $5$ are
extremal even unimodular lattices and optimal odd 
unimodular lattices, respectively.
However, 
the kissing numbers of such lattices
are $52416000$ (see \cite[Chap.~7]{SPLAG}) and 
$385024$ or $393216$
\cite{HKMV}, respectively.
This is a contradiction.
\end{proof}

% Let $C$ be a self-dual $\ZZ_{k}$-code of length $n$.
% If $n \le 47$ 
% then we have the following bound
% \[
% d_E(C) \le 2k \left\lfloor \frac{n}{24} \right\rfloor +2k,
% \]
% except when $n =23$ and $k \ge 4$, when the bound is $d_E(C) \le 3k$,
% except when $(n,k) =(22,2)$, when the bound is $d_E(C) \le 6$
% and 
% except when $(47,4)$, when the bound is $d_E(C) \le 20$.
% We say that a self-dual $\ZZ_{k}$-code meeting the bound 
% with equality is {\em extremal} for length $n \le 47$.

Hence, 
if $C$ is a self-dual $\ZZ_{k}$-code $C$  of length $n \le 48$ 
then we have the following bound:
\[
d_E(C) \le 
\begin{cases}
3k & \text{if $n =23$ and $k \ge 4$} \\
4 \lfloor \frac{n}{24} \rfloor+6
  & \text{if $n =22,46$ and $k=2$} \\  % (n,k) =(22,2) \\ 
20 & \text{if $n =47$ and $k=4$} \\  % (n,k) =(47,4) \\
2k \left\lfloor \frac{n}{24} \right\rfloor +2k & \text{otherwise.}
\end{cases}
\]
We say that a self-dual $\ZZ_{k}$-code meeting the bound 
with equality is 
%%%%%
{\em extremal}\footnote{For $k=3$, a self-dual code meeting
the bound (\ref{eq:kbound}) is usually called extremal.
However, we here adopt this definition
since we consider the existence of extremal
self-dual $\ZZ_k$-codes for all positive integers $k$
with $k \ge 2$, at once.
}
%%%%%
for length $n \le 48$.
%%%%%%%%%%%%
% For the length for which there is no extremal self-dual
% $\ZZ_k$-code, we say that a self-dual code $C$ is 
% {\em near-extremal} if $d_E(C)+k$ meets the bound.
We say that a self-dual code $C$ is 
{\em near-extremal} if $d_E(C)+k$ meets the bound.

The following lemma shows that an extremal self-dual
$\ZZ_{k}$-code of lengths $24$ and $48$ must be Type~II
for every even positive integer $k$.

\begin{lem}
\begin{itemize}
\item[\rm (a)]
Let $C$ be a Type~I $\ZZ_k$-code of length $24$.
Then $d_E(C) \le 3 k$
for every positive integer $k$ with $k \ge 2$.
\item[\rm (b)]
Let $C$ be a Type~I $\ZZ_k$-code of length $48$.
Then $d_E(C) \le 5 k$
for every positive integer $k$ with $k \ge 2$.
\end{itemize}
\end{lem}
\begin{proof}
We give the proof of (b).
% By the bounds in \cite{MPS, Rains},
By the bound (\ref{eq:kbound}),
it is sufficient to consider only $k \ge 4$.
Assume that $k \ge 4$ and $d_E(C) = 6 k$.
If $k \ge 6$ then $A_k(C)$ has minimum norm $6$.
Hence, $A_k(C)$ must be even, that is, $C$ is Type~II,
which is a contradiction.
Suppose that $k=5$.
Then $A_5(C)$ is an optimal odd unimodular lattice with
kissing number $96$,
%%%%%%%%%%%
which contradicts that
the kissing number is $385024$ or $393216$ \cite{HKMV}.
%%%%%%%%%%%
Finally, suppose that $k=4$.
Since $d_E(C) = 24$, 
$A_4(C)$ satisfies the condition that
$\min(A_4(C))=4$, the kissing number is $96$ and there is no
vector of norm $5$.
% Conway and Sloane \cite{CS-odd} show that
% when the theta series of an odd unimodular lattice $L$
% is written as
% $
% \theta_L(q)=
%  \sum_{j =0}^{\lfloor n/8\rfloor} a_j\theta_3(q)^{n-8j}\Delta_8(q)^j,
% $
% the theta series of the shadow $S$
% can be written as
% $
% \theta_S(q)= \sum_{j=0}^{\lfloor n/8\rfloor}
% \frac{(-1)^j}{16^j} a_j\theta_2(q)^{n-8j}\theta_4(q^2)^{8j},
% $
% where
% $\Delta_8(q) = q \prod_{m=1}^{\infty} (1 - q^{2m-1})^8(1-q^{4m})^8$,
% and $\theta_2(q), \theta_3(q)$ and $\theta_4(q)$ are the Jacobi
% theta series \cite{SPLAG}.
By \cite[(2) and (3)]{CS-odd}, one can determine the 
possible theta series of 
$A_4(C)$ and its shadow $S$ as follows:
\[
\begin{cases}
\theta_{A_4(C)}(q) &=
1 + 96 q^4 + (35634176 + 16777216 \alpha)q^6 
% +(805306368 -805306368 \alpha)q^7 
+ \cdots, \\
\theta_{S}(q) &=
\alpha + (96 - 96 \alpha )q^2 +(- 4416 + 4512 \alpha) q^4 + \cdots,
\end{cases}
\]
respectively, where $\alpha$ is an integer.
From the coefficients of $\theta_{S}(q)$,
% the theta series of the shadow 
it follows that $\alpha=1$.
Hence, since $S$ contains the zero-vector,
$A_4(C)$ must be  even, that is, $C$ is Type~II.

The proof of (a) is similar to that of (b), and
it can be completed more easily.
So the proof is omitted.
\end{proof}

% \bigskip
% \noindent
% {\bf This can be applied to lengths up to 71, if we can
% determine the possible theta series of unimodular lattices 
% with minimum norm 5,6.
% }
% \bigskip

The odd Leech lattice contains a $k$-frame for 
every positive integer $k$ with $k \ge 3$~\cite{Miezaki}.
The binary odd Golay code is a near-extremal Type~I code of length $24$.
Hence, there is a near-extremal Type~I $\ZZ_k$-code of length $24$
for every positive integer $k$ with $k \ge 2$.

%%%%
\subsection{Negacirculant matrices}
\label{subsec:M}

An  $n \times n$ matrix $M$ is {\em negacirculant} if
$M$ has the following form:
\[
\left( \begin{array}{ccccc}
r_0     &r_1     & \cdots &r_{n-1} \\
-r_{n-1}&r_0     & \cdots &r_{n-2} \\
-r_{n-2}&-r_{n-1}& \cdots &r_{n-3} \\
\vdots  & \vdots && \vdots\\
-r_1    &-r_2    & \cdots&r_0
\end{array}
\right).
\]
Most of matrices constructed in this paper 
are based on negacirculant matrices.
In Section \ref{sec:main}, 
in order to construct self-dual $\ZZ_k$-codes of length $4n$,
we consider a generator 
matrix  of the following form:
\begin{equation} \label{eq:GM}
\left(
\begin{array}{ccc@{}c}
\quad & {\Large I_{2n}} & \quad &
\begin{array}{cc}
A & B \\
-B^T & A^T
\end{array}
\end{array}
\right),
\end{equation}
where $A$ and $B$ are $n \times n$ negacirculant matrices,
$A^T$ denotes the transpose of the matrix $A$
and 
$I_k$ denotes the identity matrix of order $k$.
It is easy to see that the code is self-dual if
$AA^T+BB^T=-I_n$.

% \bigskip\noindent
% {\bf Will we comment on a subgroup of the automorphism
% group?}
% \bigskip

In Section \ref{sec:frame}, 
in order to find $k$-frames in some lattices, 
we need to construct matrices $M$ satisfying 
the condition (\ref{eq:condition}) 
in Proposition \ref{prop:const}.
Suppose that $p$ is a prime and $p \equiv 3 \pmod 4$.
Let $Q_{p}=(q_{ij})$ be a $p$ by $p$ matrix where
$q_{ij}=0$ if $i=j$, 
$-1$ if $j-i$ is a nonzero square $\pmod p$, and 
$1$ otherwise.
We consider the following matrix:
\[
P_{p+1} = 
\left(\begin{array}{ccccccc}
 0       & 1  & \cdots      &1  \\
 -1      & {} & {}          &{} \\
 \vdots  & {} &  Q_{p}  &{} \\
 -1      & {} & {}          &{} \\
\end{array}\right).
\]
Then it is well known that 
$P_{p+1}P_{p+1}^T=pI_{p+1}$ and 
$P_{p+1}^T=-P_{p+1}$, and 
$P_{p+1}+I_{p+1}$ is a Hadamard matrix, which
is equivalent to the Paley Hadamard matrix of order $p+1$.
Hence, these matrices satisfy (\ref{eq:condition}).
In Section \ref{sec:frame}, we 
construct more matrices $M$ satisfying (\ref{eq:condition})
using the following form:
\begin{equation} \label{eq:M}
\left(
\begin{array}{cc}
A & B \\
-B^T & A^T
\end{array}
\right),
\end{equation}
where $A$ and $B$ are $n \times n$ negacirculant matrices.
%%%%%%

% Let $C_4(P_{p+1})$ be a $\ZZ_4$-code of length $2p+2$
% with generator matrix $G$,
% where 
% $G=
% \left(\begin{array}{cc}
%  I & P_{p+1}
% \end{array}\right)$ 
% if $p \equiv 3 \pmod 8$ and $G=\left(\begin{array}{cc}
%  I & P_{p+1}+2I
% \end{array}\right)$ if $p \equiv 7 \pmod 8$,
% and 
% the entries of the matrices are regarded as elements of $\ZZ_4$.
% Then codes $C_4(P_{p+1})$ are Type~I~\cite{Z4-C-S}.

%%%%%%%%
\subsection{Positivity of coefficients of the theta series}

It is important to study
the positivity and non-negativity of coefficients of 
the theta series of extremal unimodular lattices. 
For example, let $\sum_{m=0}^{\infty}A_{m}q^m$ be the theta series
of an even unimodular lattice in dimension $n$.
Then it was shown in~\cite{Siegel}
that the coefficient $A_{2 \lfloor \frac{n}{24} \rfloor+2}$ 
is always positive when
$A_2=A_4=\cdots =A_{2\lfloor \frac{n}{24} \rfloor} =0$
(see also \cite{MOS75}).
% Hence, 
% the minimum norm $\min(L)$ of $L$ is bounded by
% \[
% \min(L) \le 2 \left\lfloor \frac{n}{24} \right\rfloor+2.
% \]
This gives 
the upper bound of the minimum norm 
of even unimodular lattices as in Section \ref{sec:2U}. 
%Then it was shown in \cite{RS-bound} that 
%$A_{2 \lfloor \frac{n}{24} \rfloor+2}$ is positive 
%if $A_i=0$ ($i=1,2,\ldots,2 \lfloor \frac{n}{24} \rfloor+1$)
%unless $n=23$. 
%This result gives {\bf ?????}
%the upper bound of the minimum norm 
%of unimodular lattices as in Section \ref{sec:2U}. 

To discuss the
positivity of coefficients of the theta series of 
extremal (optimal) unimodular lattices 
listed in Table \ref{Tab:L},
the following lemma is used.

\begin{lem}\label{lem:T}
Let $L$ be a unimodular lattice in dimension $n$
with theta series $\sum_{m=0}^{\infty}A_{m}q^m$.
If $L$ contains a $k$-frame then $A_k \ge 2n$.
\end{lem}

\begin{rem}
As described in Section \ref{Sec:1}, 
the odd Leech lattice contains a $k$-frame for 
every positive integer $k$ with $k \ge 3$~\cite{Miezaki}.
By the above lemma, $A_{m} \ge 48$
for every positive integer $m$ with $m \ge 3$.
\end{rem}

\begin{rem}
At dimensions $n=20, 28, 32$, 
there are other unimodular lattices with
the same theta series as 
one of the unimodular lattices listed in Table \ref{Tab:L}. 
Of course, it also holds that
$A_{m} \ge 2n$
%% for every positive integer $m$ with $m \ge 2\lfloor \frac{n}{24} \rfloor+2$ 
%% unless $n = 28$ when $m \geq 3$. 
for every positive integer $m \ge 2,3,4$ if $n=20,28,32$, 
respectively, for the other lattices.
\end{rem}

% For every positive integer $m$ with
% $m\geq \lfloor \frac{n}{24}\rfloor+1$,
% the shell of norm $2m$ of 
% an extremal even unimodular lattice in 
% dimension $n=32,40,48,56,64$ forms 
% a spherical $t$-design explicitly,
% where $t=7,3,11,7,3$, respectively.

%% We consider the positivity of coefficients in more detail.
%% %% Since $D_{12}^+$ (resp.\ $D_8^2$) is the unique extremal
%% %% odd unimodular lattice in dimension $12$ (resp.\ $16$),
%% %% Corollary \ref{cor:posi} shows that $A_m \ge 24$ $(m \ge 2)$ 
%% %% (resp.\ $A_m \ge 32$ $(m \ge 2)$)
%% %% for any extremal odd unimodular lattice in dimension $12$ (resp.\ $16$).
%% %%%%%%
%% At dimension $20$, extremal odd unimodular lattices are divided into
%% ten kinds of theta series (see \cite[Table 16.7]{SPLAG}).
%% Corollary \ref{cor:posi} shows that $A_m \ge 40$
%% $(m \ge 2)$ for the lattices with theta series $1+760 q^2 + \cdots$
%% and $1+120 q^2 + \cdots$.
%% %%%%%%
%% At dimension $28$,
%% there are $38$ non-isomorphic optimal odd unimodular lattices, 
%% and $36$ lattices have 
%% theta series $1 + 2240 q^3 + \cdots$ and two lattices have
%% theta series $1 + 1728 q^3 + \cdots$ \cite{BV}.
%% Corollary \ref{cor:posi} shows that $A_m \ge 56$
%% $(m \ge 3)$ for the $36$ lattices.
%% %%%%%
%% At dimension $32$,
%% there are $5$ non-isomorphic extremal odd unimodular lattices
%% and  the five lattices have the identical
%% theta series \cite{CS-odd}.
%% Corollary \ref{cor:posi} shows that $A_m \ge 64$
%% $(m \ge 4)$  for the $5$ lattices.
%% 

%%%%%%%%%%%%%%%%%%%%%%%%%%%%%%%%%%%%%
\section{Number theoretical results}\label{sec:NT}

Let $k,\ell$ and $m$ be positive integers with 
$k \ge 2$ and $\ell \le k-1$.
Consider the following lattice in dimension $4$:
\begin{equation}\label{eq:L}
L_{\ell,m,k}=\{(a, b, c, d)\in\ZZ^4
\mid b \equiv c-\ell d \!\!\!\pmod{k} \text{ and } 
d \equiv a+\ell b \!\!\!\pmod{k}\}, 
\end{equation}
where we consider the inner product $\langle x,y \rangle$
induced by 
$(a^2+mb^2+c^2+md^2)/k$, instead of the standard inner product.
The theta series $\theta_{L_{\ell,m,k}}(q)$ of $L_{\ell,m,k}$ is 
$\sum_{x \in L_{\ell,m,k}} q^{\langle x,x \rangle}$.

\begin{lem}\label{lem:lattice}
If $m+\ell^2 \equiv -1 \pmod{k}$
then
$\theta_{L_{\ell,m,k}}(z)$ 
is a modular form (of weight $2$) for $\Gamma_0(4m)$, 
where $q=e^{2\pi i z}$, $z$ is in the upper half plane and 
\[
\Gamma_0(N)=
\left\{
\begin{pmatrix}
a&b\\
c&d
\end{pmatrix}
\in SL_2(\ZZ) \mid   c\equiv 0\pmod{N}
\right\}.
\] 
\end{lem}
\begin{proof}
The lattice $L_{\ell,m,k}$ is spanned by 
$(k,0,0,0)$, $(0,0,k,0)$, $(1,0,\ell+1,1)$ and $(0,1,\ell^2+1,\ell)$
with Gram matrix:
\[
M=
\begin{pmatrix}
k& 0& 1& 0 \\
0& k& \ell& \ell^2 + 1 \\
1& \ell& (\ell^2 + m + 1)/k& (\ell^2 + m + 1)\ell/k \\
0& \ell^2 + 1& (\ell^2 + m + 1)\ell/k& (\ell^2 + m + 1)(\ell^2 + 1)/k
\end{pmatrix}.
\]
Since 
\[
m M^{-1}=
\begin{pmatrix}
 (1+\ell^2+m)/k & 0            & -1-\ell^2  & \ell \\
 0            & (1+\ell^2+m)/k& 0         & -1\\
 -1-\ell^2     & 0            &k+k \ell^2& -k\ell\\
 \ell            & -1           & -k \ell      & k 
\end{pmatrix}
\]
and  $m+\ell^2 \equiv -1 \pmod{k}$,
$m M^{-1}$ has integer entries.
Since $\det M=m^2$, 
$\theta_{L_{\ell,m,k}}(z)$ is a modular form (of weight $2$) for 
$\Gamma_0(4m)$~\cite[Corollary~4.9.2]{Miyake}. 
\end{proof}

In order to give infinite families of $k$-frames 
by Proposition \ref{prop:const},
we derive the following theorem.
Its proof is similar to that in \cite{Chapman, HM12, Miezaki},
but this is more complicated.
%% It needs some facts of modular forms for congruence subgroups. 
Our notation and terminology for modular forms
follow from~\cite{Miyake}
(see~\cite{Miyake} for undefined terms).

%%%%%%%%%%%%%%%%%%%%%%%%%%%%%%%%

\begin{thm}\label{thm:prime}
\begin{itemize}
\item[\rm (a)]\label{thm:1}
There are integers $a,b,c$ and $d$ satisfying
$b \equiv c-d \pmod 3$,
$d \equiv a+b \pmod 3$ and 
$p=\frac{1}{3}(a^2+25b^2+c^2+25d^2)$ for each prime $p \ne 2, 5, 7, 13, 23$.

\item[\rm (b)] \label{thm:2}
There are integers $a,b,c$ and $d$ satisfying
%% $c \equiv 2a+b \pmod 4$,
%% $d \equiv a+2b \pmod 4$ and 
$b \equiv c-2d \pmod 4$,
$d \equiv a+2b \pmod 4$ and 
$p=\frac{1}{4}(a^2+7b^2+c^2+7d^2)$ for each prime $p \ne 2,7$.

\item[\rm (c)] \label{thm:7}
There are integers $a,b,c$ and $d$ satisfying
$b \equiv c \pmod 5$,
$d \equiv a \pmod 5$ and 
$p=\frac{1}{5}(a^2+49b^2+c^2+49d^2)$ for each prime $p \ne 2, 3, 7, 11,19,29$.

\item[\rm (d)] \label{thm:6}
There are integers $a,b,c$ and $d$ satisfying
$b \equiv c-2d \pmod 5$,
$d \equiv a+2b \pmod 5$ and 
$p=\frac{1}{5}(a^2+25b^2+c^2+25d^2)$ for each prime $p \ne 2, 3, 17$.

\item[\rm (e)] \label{thm:3}
There are integers $a,b,c$ and $d$ satisfying
%% $c \equiv 2a+b \pmod 4$,
%% $d \equiv a+2b \pmod 4$ and 
$b \equiv c-2d \pmod 4$,
$d \equiv a+2b \pmod 4$ and 
$p=\frac{1}{4}(a^2+15b^2+c^2+15d^2)$ for each prime $p \ne 2,3$.

\item[\rm (f)] \label{thm:4}
There are integers $a,b,c$ and $d$ satisfying
$b \equiv c-2d \pmod 6$,
$d \equiv a+2b \pmod 6$ and 
$p=\frac{1}{6}(a^2+49b^2+c^2+49d^2)$ for each prime $p \ne 2, 3, 5,7$.

\item[\rm (g)] \label{thm:5}
There are integers $a,b,c$ and $d$ satisfying
$b \equiv c \pmod 4$,
$d \equiv a \pmod 4$ and 
$p=\frac{1}{4}(a^2+19b^2+c^2+19d^2)$ for each prime $p \ne 2, 3, 13, 19$.

\item[\rm (h)] \label{thm:8}
There are integers $a,b,c$ and $d$ satisfying
$b \equiv c \pmod 5$,
$d \equiv a \pmod 5$ and 
$p=\frac{1}{5}(a^2+39b^2+c^2+39d^2)$ for each prime $p \ne 2, 3, 7, 17$.

\end{itemize}
\end{thm}
\begin{proof}
%% Since all cases are similar, 
We only give details for Case (a), to save space.
The ideas of the proofs of the other cases
are similar to that of Case (a),
which is the most complicated case,
where main different parts are mentioned in 
Tables \ref{Tab:N1}--\ref{Tab:N6}.
% The proofs for the other cases
% are left to the reader, but
%%For the convenience of the reader,
% the main differences are given in Tables 
% \ref{Tab:N1}--\ref{Tab:N6}.

% The proof is left as an exercise for the reader.
% The main ideas of the proof are similar to Example 1.1.7 on Page
% 2 and the difference is that the proof here needs to show that the
% limit function is continuous. 

Consider the lattice $L_{1,25,3}$ given in (\ref{eq:L}).
We have verified by {\sc Magma} that 
it has the following theta series: 
\begin{align*}
\theta_{L_{1,25,3}}(q)&
=1 + 4q^3 + 4q^6 + 4q^9 + 8q^{10} + 4q^{11} + \cdots 
% \label{eqn:theta}\\
=\sum_{n=0}^{\infty}a(n)q^n\ (\text{say}). \nonumber
\end{align*}
By Lemma \ref{lem:lattice}, $\theta_{L_{1,25,3}}(z)$
is a modular form for $\Gamma_0(100)$,
where $q=e^{2\pi i z}$, $z$ is in the upper half plane 
(see Table \ref{Tab:N1} for the other cases).

%%%%%%%%%%%%%%%%%%%%%%%%%%%%%%%%%%
\begin{table}[thb]
\caption{$\Gamma_0(N)$, $\dim(S_2(\Gamma_0(N)))$  and 
genera  of $\Gamma_0(4N)$}
\label{Tab:N1}
\begin{center}
%{\small
{\footnotesize
%{\scriptsize
\begin{tabular}{c|c|c|c} 
\noalign{\hrule height0.8pt}
Case & \multicolumn{1}{|c|}{$\Gamma_0(N)$} & $\dim(S_2(\Gamma_0(N)))$ &
Genus of $\Gamma_0(4N)$ \\
\hline
(b) & $\Gamma_0(28)$ & $2$ & $11$ \\\hline
(c) & $\Gamma_0(196)$ & $17$ & $89$ \\\hline
(d) & $\Gamma_0(100)$ & $7$ & $43$ \\\hline
(e) & $\Gamma_0(60)$ & $7$ & $37$ \\\hline
(f) & $\Gamma_0(196)$ & $17$ & $89$ \\\hline
(g) & $\Gamma_0(76)$ & $8$ & $35$ \\\hline
(h) & $\Gamma_0(156)$ & $23$ & $101$ \\
\noalign{\hrule height0.8pt}
\end{tabular}
}
\end{center}
\end{table}
%%%%%%%%%%%%%%%%%%%%%%%%%%%%%

We denote by $S_2(\Gamma_0(N))$ 
the space of cusp forms of weight $2$ for $\Gamma_0(N)$. 
It is known that 
$\dim(S_2(\Gamma_0(100)))$ 
%the dimension of the space of cusp forms of weight $2$ 
%for $\Gamma_0(100)$ 
is seven 
(see Table \ref{Tab:N1} for the other cases),
and using {\sc Magma} 
we have found some basis 
$f_1(z)$, $f_2(z)$, $f_3(z)$,
$f_4(z)$, $f_5(z)$, $f_7(z)$, $f_{9}(z)$ 
such that 
% \[
% f_i(z)=c_iq^i+c_{i, 12}q^{12}+c_{i, 13}q^{13}+\cdots, 
% \]
% where $c_i\neq 0$ and $i=1,2,\ldots, 13$. 
\[
f_i(z)=q^i+c_{i, 6}q^{6}+c_{i, 7}q^{7}+\cdots, 
\]
for $i=1,2,3,4,5$,
$f_7(z)=q^{7}+\cdots$,  and
$f_9(z)=q^{9}+\cdots$. 
%where $c_{i} \neq 0$ for $i=9$. {\bf ?}
In particular, we use $f_i(z)\ (i=1,3,5,7,9)$, 
which are explicitly written as:
\begin{align}
\left\{
\begin{array}{rl}\label{eqn:Cusp}
f_1(z)&=   q - q^{11} - q^{19} - 2q^{21} + 4q^{29} +\cdots, \\
f_3(z)&=q^3 - 2q^{13} + q^{17} - 3q^{27} + \cdots,\\
f_5(z)&=  q^5 - 2q^{15} - q^{25} + \cdots,\\
f_7(z)&= q^7 - q^{13} - 2q^{17} + 3q^{23} - q^{27} + \cdots,\\ 
f_9(z)&=  q^9 + q^{11} - 3q^{19} - 2q^{21} + 2q^{29} + \cdots  
\end{array}
\right.
\end{align}
(see Table \ref{Tab:N2} for the other cases).
For $i=1,3,5,7,9$, we denote by $c_{f_{i}}(n)$ 
the coefficient of $f_{i}(z)$ as follows: 
\[
f_{i}(z)=\sum_{n=1}^{\infty}c_{f_{i}}(n)q^n.
\]

%%%%%%%%%%%%%%%%%%%%%%%%%%%%%%%%%%
\begin{table}[thb]
\caption{Some basis $f_i(z)$ in (\ref{eqn:Cusp})}
\label{Tab:N2}
\begin{center}
%{\small
{\footnotesize
%{\scriptsize
\begin{tabular}{c|ll}
\noalign{\hrule height0.8pt}
Case & \multicolumn{2}{c}{$f_i(z)$}  \\
\hline
(b)&$f_1(z)=q - 2q^3 + q^7 + q^9 + \cdots$ & \\
\hline
(c)&$f_1(z)=   q+q^{23}-q^{29}+\cdots$&$f_3(z)= q^3-q^{27}+ \cdots$\\
& 
$f_5(z)=  q^5-2 q^{19}+q^{27}+ \cdots$&$f_7(z)= q^7-2 q^{21}+ \cdots$ \\
&
$f_9(z)= q^9-q^{23}+\cdots$&$f_{11}(z)= q^{11}+q^{23}+2 q^{29}+ \cdots$ \\
& $f_{13}(z)= q^{13}+q^{19}-2 q^{27}+ \cdots$&$
f_{15}(z)= q^{15}+q^{23}-q^{29}+ \cdots$ \\
&$f_{17}(z)= q^{17}+q^{19}-q^{27}+ \cdots$&$
f_{25}(z)= q^{25}+q^{29}+ \cdots$  \\ \hline
(d)
&$f_1(z)=   q - q^{11} +\cdots$&$f_3(z)=q^3 - 2q^{13} +\cdots$\\
& $f_5(z)=  q^5 - 2q^{15} + \cdots$&$ f_7(z)= q^7 - q^{13} + \cdots$\\ 
& $f_9(z)=  q^9 + q^{11} + \cdots$ & 
\\ \hline
(e)&$f_1(z)=   q - q^9 +\cdots$ &$f_3(z)= q^3 - 2q^{9} + \cdots$\\
&
$f_5(z)=  q^5 - 2q^{11} + \cdots$&$f_7(z)= q^7 - q^9 +\cdots 
$\\
\hline
(f)&$f_1(z)=   q+q^{23}-q^{29}+\cdots$&$f_3(z)= q^3-q^{27}+ \cdots$\\
& 
$f_5(z)=  q^5-2 q^{19}+q^{27}+ \cdots$&$f_7(z)= q^7-2 q^{21}+ \cdots$ \\
&
$f_9(z)= q^9-q^{23}+\cdots$&$f_{11}(z)= q^{11}+q^{23}+2 q^{29}+ \cdots$ \\
& $f_{13}(z)= q^{13}+q^{19}-2 q^{27}+ \cdots$&$
f_{15}(z)= q^{15}+q^{23}-q^{29}+ \cdots$ \\
&$f_{17}(z)= q^{17}+q^{19}-q^{27}+ \cdots$&$
f_{25}(z)= q^{25}+q^{29}+ \cdots$  \\ \hline
(g)&$f_1(z)=   q - 2q^9 +\cdots$&$ f_3(z)= q^3 - 2q^9 + \cdots$\\
& $f_5(z)=  q^5 - q^9  + \cdots$&$ f_7(z)= q^7 - 2q^9  +\cdots$ \\ \hline
(h)&
$f_1(z)=   q - q^{25} - q^{27} +\cdots$ & 
$f_3(z)= q^3 + q^{23} - q^{25} - q^{27} - 2q^{29} +\cdots$ \\
&$f_5(z)=  q^5 + q^{23} - 3q^{25} - 2q^{29} + \cdots$ & 
$f_7(z)= q^7 - 2q^{25} +\cdots$ \\
&$f_9(z)=  q^9 + 2q^{23} - 2q^{25} - 3q^{27} - 2q^{29} + \cdots$ & 
$f_{11}(z)= q^{11} - q^{27} + q^{29} + \cdots$ \\
&$f_{13}(z)= q^{13} - 2q^{25} - 6q^{29} + \cdots$ & 
$f_{15}(z)= q^{15} + q^{23} - q^{25} - q^{27} - q^{29} + \cdots$ \\
&$f_{17}(z)= q^{17} + q^{23} - 2q^{25} - q^{27} - 2q^{29} + \cdots$ & 
$f_{19}(z)=  q^{19} - 2q^{25} - 2q^{29} +\cdots$ \\
&$f_{21}(z)= q^{21} + q^{27} - q^{29} + \cdots$ &
$f_{23}(z)=  3q^{23} - 3q^{25} - 2q^{27} - 2q^{29} + \cdots$ 
\\ \noalign{\hrule height0.8pt}
\end{tabular}
}
\end{center}
\end{table}
%%%%%%%%%%%%%%%%%%%%%%%%%%%%%

%%%%%%%%%%%%%%%%%%%%%%%%%%%%%%%%%%%%%
\begin{table}[thb]
\caption{New modular form $h_{N}(z)$ in (\ref{eqn:Hecke})}
\label{Tab:N3}
\begin{center}
%{\small
{\footnotesize
%{\scriptsize
\begin{tabular}{c|l}
\noalign{\hrule height0.8pt}
Case & \multicolumn{1}{c}{$h_{N}(z)$} \\
\hline
(b) & $\frac{4}{3}
{(}
\frac{\eta(4z)^8}{\eta(2z)^4}
-\frac{\eta(28z)^8}{\eta(14z)^4}
-7f_1(z)
{)}$ \\
\hline
(c)  & $\frac{4}{21}
{(}
\frac{\eta(4z)^8}{\eta(2z)^4}
-8\frac{\eta(28z)^8}{\eta(14z)^4}
+49\frac{\eta(196z)^8}{\eta(98z)^4}
-f_1(z)
-4f_3(z)
+15f_5(z)
%+0f_7(z)
-13f_9(z)$ \\ 
& $-12f_{11}(z)
+7f_{13}(z)
+18f_{15}(z)
+3f_{17}(z)
+11f_{25}(z)
{)}
$ \\ 
\hline
(d)  & $\frac{4}{15}
{(}
\frac{\eta(4z)^8}{\eta(2z)^4}
+4\frac{\eta(20z)^8}{\eta(10z)^4}
+25\frac{\eta(100z)^8}{\eta(50z)^4}
-f_1(z)
-4f_3(z)
+5f_5(z)
+7f_7(z)
+2f_9(z)
{)}
$\\
\hline 
(e) &  $\frac{2}{3}
{(}
\frac{\eta(4z)^8}{\eta(2z)^4}
-3\frac{\eta(12z)^8}{\eta(6z)^4}
+5\frac{\eta(20z)^8}{\eta(10z)^4}
-15\frac{\eta(60z)^8}{\eta(30z)^4}
-f_1(z)
-f_3(z)
+f_5(z)
+4f_7(z)
{)}
$ \\
\hline
(f)  & $\frac{4}{21}
{(}
\frac{\eta(4z)^8}{\eta(2z)^4}
-8\frac{\eta(28z)^8}{\eta(14z)^4}
+49\frac{\eta(196z)^8}{\eta(98z)^4}
-f_1(z)
-4f_3(z)
-6f_5(z)
%+0f_7(z)
+8f_9(z)$ \\
& $+9f_{11}(z)
+7f_{13}(z)
-3f_{15}(z)
+3f_{17}(z)
-10f_{25}(z)
{)}
$ \\
\hline
(g) &  $\frac{4}{9}
{(}
\frac{\eta(4z)^8}{\eta(2z)^4}
-19\frac{\eta(76z)^8}{\eta(38z)^4}
-f_1(z)
-4f_3(z)
+3f_5(z)
+f_7(z)
{)}
$ \\
\hline
(h)  & 
$\frac{2}{7}
{(}
\frac{\eta(4z)^8}{\eta(2z)^4}
-3\frac{\eta(12z)^8}{\eta(6z)^4}
+13\frac{\eta(52z)^8}{\eta(26z)^4}
-39\frac{\eta(156z)^8}{\eta(78z)^4}
-f_1(z)
-f_3(z)
+8f_5(z)
-8f_7(z)$\\
&$-f_9(z)
+2f_{11}(z)
+f_{13}(z)
+8f_{15}(z)
-18f_{17}(z)
+8f_{19}(z)
-8f_{21}(z)
+3f_{23}(z)
{)}
$ \\ 
\noalign{\hrule height0.8pt}
\end{tabular}
}
\end{center}
\end{table}
%%%%%%%%%%%%%%%%%%%%%%%%%%%%%%%%%%%%%

Let 
%\begin{align*}
$\eta(z)=q^{\frac{1}{24}}\prod_{n=1}^{\infty}(1-q^n)$
%\end{align*}
be the Dedekind $\eta$-function. 
%Then $\eta(4z)^8/\eta(2z)^4$ 
%is a modular form for $\Gamma_0(4)$~\cite[p.~145, Problem~10]{Kob}. 
Then 
\[
\frac{\eta(4z)^8}{\eta(2z)^4}=\sum_{n=1}^{\infty}\sigma_1(2n-1)q^{2n-1} 
\]
is a modular form for $\Gamma_0(4)$, where 
$\sigma_1(n)=\sum_{m|n}m$ (see~\cite[p.~145, Problem~10]{Kob}). 
We define a new modular form $h_{100}(z)$ for $\Gamma_0(100)$ as follows 
(see Table \ref{Tab:N3} for the other cases): 
\begin{multline}\label{eqn:Hecke}
h_{100}(z)=\frac{4}{15}\Bigg{(}
\frac{\eta(4z)^8}{\eta(2z)^4}
+4\frac{\eta(20z)^8}{\eta(10z)^4}
+25\frac{\eta(100z)^8}{\eta(50z)^4}
-f_1(z)
+11f_3(z)\\
-10f_5(z)
-8f_7(z)
+2f_9(z)
\Bigg{)}
=\sum_{n=0}^{\infty}b(n)q^n\ (\text{say}).
\end{multline}
%is a modular form for $\Gamma_0(92)$. 
Note that 
all the degrees of $\eta(4nz)^8/\eta(2nz)^4$ 
are divided by $n$, namely, $q^n+4q^{2n}+\cdots$. 
Hence, 
%we have 
%$b_1(7)=0$ and 
$b(p)=\frac{4}{15}(\sigma_1(p)-c_{f_1}(p)+11c_{f_3}(p)
-10c_{f_5}(p)
-8c_{f_7}(p)
+2c_{f_9}(p))
%=\frac{4}{15}(p+1-c_{f_1}(p)+11c_{f_3}(p)
%-10c_{f_5}(p)
%-8c_{f_7}(p)
%+2c_{f_9}(p))
$ 
for each odd prime $p$ with $p\neq 5$, 
%while 
noting that $b(5)=0$.

Let
\[
\chi_2(n)=
\left\{
\begin{array}{cl}
0 & \text{ if }n\equiv 0\pmod{2}\\
1 & \text{ otherwise. }
\end{array}
\right.
\]
Then
\begin{align*}
%\label{eqn:theta} 
(\theta_{L_{1,25,3}}(z))_{\chi_2}&=\sum_{n=0}^{\infty}\chi_2(n)a(n)q^n
= 4 q^3 + 4 q^9 + 4 q^{11} + \cdots 
\text{ and} \\
%\label{eqn:theta} 
(h_{100}(z))_{\chi_2}&=\sum_{n=0}^{\infty}\chi_2(n)b(n)q^n
= 4 q^3 + 4 q^9 + 4 q^{11} +  \cdots 
\end{align*}
are modular forms with character $\chi_2$
for $\Gamma_0(400)$~{\cite[p.~127, Proposition~17]{Kob}}.
Using Theorem~7 in~\cite{Murty} and the fact that 
the genus of $\Gamma_0(400)$ is $43$ 
(see Table \ref{Tab:N1} for the other cases), 
the verification by {\sc Magma} that 
$\chi_2(n)a(n)=\chi_2(n)b(n)$ for $n\leq 43\times 2$ 
shows
\[
(\theta_{L_{1,25,3}}(z))_{\chi_2}=(h_{100}(z))_{\chi_2}. 
\]
Hence, for each odd prime $p$ with $p\ne 5$, we have
\begin{align}\label{eqn:coef56}
a(p)=\frac{4}{15}(p+1-(c_{f_1}(p)-11c_{f_3}(p)
+10c_{f_5}(p) +8c_{f_7}(p)-2c_{f_9}(p))),
\end{align}
(see Table \ref{Tab:N4} for the other cases and
the following paragraph is unnecessary
for the other cases).

Now take the unique normalized cusp form 
$f^{\prime}_5(z)=\sum_{n=1}^{\infty}c_{f^{\prime}_5}(n)q^n 
\in S_2(\Gamma_0(20))$.
The verification by {\sc Magma} that 
$c_{f_5}(5n)=c_{f^{\prime}_5}(n)$ for $n\leq 43\times 2$
shows that $f_5(z)=f^{\prime}_5(5z)$. 
Thus, for each prime $p$ with $p\neq 5$, $c_{f_5}(p)=0$. 
Hence, for each odd prime $p$ with $p\neq 5$, we have 
\begin{align}\label{eqn:coef56-2}
a(p)=\frac{4}{15}(p+1-(c_{f_1}(p)-11c_{f_3}(p)
+8c_{f_7}(p)-2c_{f_9}(p))).
\end{align}

%% It is known that the dimension of the space of cusp forms of weight $2$ 
%% for $\Gamma_0(20)$ is one and using {\sc Magma} 
%% we have found some basis such that 
%% $f^\prime_5(z)=\sum_{n=1}^{\infty}c_{f^{\prime}_5}(n)q^n= q - 2q^3 - q^5 + 2q^7 + q^9 +\cdots$. 
%% Using Theorem~7 in~\cite{Murty} and the fact that 
%% the genus of $\Gamma_0(400)$ is $43$, 
%% the verification by {\sc Magma} that 
%% $f_5(z)=c_{f^{\prime}_5}(n)$ for $n\leq 86$ 
%% shows
%% \[
%% f_5(z)=f^\prime_5(5z). 
%% \]
%% Hence, for prime $p\neq 5$, $c_{f_5}(p)=0$ we have 
%% \begin{align}\label{eqn:coef56}
%% a(p)=\frac{4}{15}(p+1-(c_{f_1}(p)-11c_{f_3}(p)
%% +8c_{f_7}(p)-2c_{f_9}(p)), 
%% \end{align}
%% (see Table \ref{Tab:N4} for the other cases). 

%%%%%%%%%%%%%%%%%%%%%%%%%%%%%%%%%%%%%
\begin{table}[thbp]
\caption{$a(p)$ in (\ref{eqn:coef56})}
\label{Tab:N4}
\begin{center}
%{\small
{\footnotesize
%{\scriptsize
\begin{tabular}{c|l}
\noalign{\hrule height0.8pt}
Case &\multicolumn{1}{c}{ $a(p)$} \\
\hline
(b) & $\frac{4}{3}(p+1-(c_{f_1}(p)))
$\\
\hline
(c) & $\frac{4}{21}(p+1-(c_{f_1}(p)+4c_{f_3}(p)-15c_{f_5}(p) 
+13c_{f_9}(p)+12c_{f_{11}}(p)$ \\
& $-7c_{f_{13}}(p)
-18c_{f_{15}}(p)-3c_{f_{17}}(p)-11c_{f_{25}}(p))) 
$ \\ 
\hline
(d) & $\frac{4}{15}(p+1-(c_{f_1}(p)+4c_{f_3}(p)
-5c_{f_5}(p) -7c_{f_7}(p)-2c_{f_9}(p))
$\\ 
\hline 
(e) & $\frac{2}{3}(p+1-(c_{f_1}(p)+c_{f_3}(p)-c_{f_5}(p) -4c_{f_7}(p))) 
$\\
\hline
(f) & $\frac{4}{21}(p+1-(c_{f_1}(p)+4c_{f_3}(p)+6c_{f_5}(p) 
-8c_{f_9}(p)-9c_{f_{11}}(p)$ \\
& $-7c_{f_{13}}(p)
+3c_{f_{15}}(p)-3c_{f_{17}}(p)+10c_{f_{25}}(p))) 
$\\
\hline
(g) & $\frac{4}{9}(p+1-(c_{f_1}(p)+4c_{f_3}(p)-3c_{f_5}(p) -c_{f_7}(p)))
$\\
\hline
(h) & 
$\frac{2}{7}(p+1-(c_{f_1}(p)+c_{f_3}(p)
-8c_{f_5}(p) +8c_{f_7}(p)+c_{f_9}(p)-2c_{f_{11}}(p)$ \\
&$-c_{f_{13}}(p)-8c_{f_{15}}(p)+18c_{f_{17}}(p)
-8c_{f_{19}}(p)
+8c_{f_{21}}(p)
-3c_{f_{23}}(p)) 
$ \\
\noalign{\hrule height0.8pt}
\end{tabular}
}
\end{center}
\end{table}
%%%%%%%%%%%%%%%%%%%%%%%%%%%%%%%%%%%%%

\begin{landscape}
%%%%%%%%%%%%%%%%%%%%%%%%%%%%%%%%%%%%%%%%%%%%%%%%
\begin{table}[thbp]
\caption{Matrices in (\ref{mat:Hecke})}
\label{Tab:N5}
\begin{center}
%{\small
%{\footnotesize
%{\scriptsize
{\tiny
\begin{tabular}{c|l}
\noalign{\hrule height0.8pt}
Case  &  \multicolumn{1}{c}{Matrices}\\
\hline
(b) & $\begin{pmatrix}
1
\end{pmatrix}
$\\
\hline
(c) & $\begin{pmatrix}
 1 & -2 & 0 & \frac{1}{2} \left(1-\sqrt{29}\right) & 1 & 0 & -4 & 0 & 6 & -5 \\
 1 & -2 & 0 & \frac{1}{2} \left(1+\sqrt{29}\right) & 1 & 0 & -4 & 0 & 6 & -5 \\
 1 & 0 & 0 & 0 & -3 & 4 & 0 & 0 & 0 & 5 \\
 1 & -1 & -3 & 0 & -2 & -3 & -2 & 3 & -3 & 4 \\
 1 & 1 & 3 & 0 & -2 & -3 & 2 & 3 & 3 & 4 \\
 1 & -\sqrt{2} & 2 \sqrt{2} & 0 & -1 & -2 & 0 & -4 & -\sqrt{2} & 3 \\
 1 & \sqrt{2} & -2 \sqrt{2} & 0 & -1 & -2 & 0 & -4 & \sqrt{2} & 3 \\
 1 & 2 & 0 & 0 & 1 & 0 & 4 & 0 & -6 & -5 \\
 1 & 2 \sqrt{2} & -\sqrt{2} & 0 & 5 & 4 & -3 \sqrt{2} & -4 & -\sqrt{2} & -3 \\
 1 & -2 \sqrt{2} & \sqrt{2} & 0 & 5 & 4 & 3 \sqrt{2} & -4 & \sqrt{2} & -3
\end{pmatrix}
$\\
\hline
(d) & $\begin{pmatrix}
 1 & -1  & -2 & -2 \\
 1 & 1  & 2 & -2 \\
 1 & 2  & -2 & 1 \\
 1 & -2  & 2 & 1 
\end{pmatrix}
$\\
\hline
(e) & $\begin{pmatrix}
 1 & 1 & -1 & -4  \\
 1 & -1 & 1 & 0 \\
 1 & -1-\sqrt{-2} & -1 & 2 \\
 1 & -1+\sqrt{-2} & -1 & 2 
\end{pmatrix}
$\\
\hline
(f) & $\begin{pmatrix}
 1 & -2 & 0 & \frac{1}{2} \left(1-\sqrt{29}\right) & 1 & 0 & -4 & 0 & 6 & -5 \\
 1 & -2 & 0 & \frac{1}{2} \left(1+\sqrt{29}\right) & 1 & 0 & -4 & 0 & 6 & -5 \\
 1 & 0 & 0 & 0 & -3 & 4 & 0 & 0 & 0 & 5 \\
 1 & -1 & -3 & 0 & -2 & -3 & -2 & 3 & -3 & 4 \\
 1 & 1 & 3 & 0 & -2 & -3 & 2 & 3 & 3 & 4 \\
 1 & -\sqrt{2} & 2 \sqrt{2} & 0 & -1 & -2 & 0 & -4 & -\sqrt{2} & 3 \\
 1 & \sqrt{2} & -2 \sqrt{2} & 0 & -1 & -2 & 0 & -4 & \sqrt{2} & 3 \\
 1 & 2 & 0 & 0 & 1 & 0 & 4 & 0 & -6 & -5 \\
 1 & 2 \sqrt{2} & -\sqrt{2} & 0 & 5 & 4 & -3 \sqrt{2} & -4 & -\sqrt{2} & -3 \\
 1 & -2 \sqrt{2} & \sqrt{2} & 0 & 5 & 4 & 3 \sqrt{2} & -4 & \sqrt{2} & -3
\end{pmatrix}
$\\
\hline
(g) & $\begin{pmatrix}
 1 & -1 & -4 & 3  \\
 1 & 2 & -1 & -3 \\
 1 & 1 & 0 & -1 \\
 1 & -2 & 3 & -1 
\end{pmatrix}
$\\
\hline
(h) & $\left(
\begin{array}{cccccccccccc}
 1 & -1 & 2 & 4 & 1 & -4 & 1 & -2 & 2 & -8 & -4 & -1 \\
 1 & -1 & 2 & -4 & 1 & 4 & 1 & -2 & 2 & 0 & 4 & -1 \\
 1 & -1 & -4 & -2 & 1 & -4 & 1 & 4 & 2 & -2 & 2 & -1 \\
 1 & 1 & 0 & 2 & 1 & 0 & 1 & 0 & -6 & 2 & 2 & 1 \\
 1 & -\sqrt{-3} & 2 & -2 & -3 & -2 & -1 & -2 \sqrt{-3} & 6 & -6 & 2  \sqrt{-3} & 2+ \sqrt{-3} \\
 1 &  \sqrt{-3} & 2 & -2 & -3 & -2 & -1 & 2 \sqrt{-3} & 6 & -6 & -2  \sqrt{-3} & 2- \sqrt{-3} \\
 1 & -\frac{1}{2}  \left(3 +\sqrt{-3}\right) & -1 & 1 & \frac{3}{2} \left(1+ \sqrt{-3}\right) & -2 & -1 & \frac{1}{2} \left(3+ \sqrt{-3}\right) & -3 & 6 & -\frac{1}{2}  \left(3 +\sqrt{-3}\right) & -1- \sqrt{-3} \\
 1 & \frac{1}{2}  \left(-3 +\sqrt{-3}\right) & -1 & 1 & \frac{3}{2} \left(1- \sqrt{-3}\right) & -2 & -1 & \frac{1}{2} \left(3- \sqrt{-3}\right) & -3 & 6 & \frac{1}{2}  \left(-3 +\sqrt{-3}\right) & -1+ \sqrt{-3} \\
 1 & \frac{1}{2} \left(1- \sqrt{-11}\right) & -3 & -1 & -\frac{1}{2}  \left(5 +\sqrt{-11}\right) & 6 & 1 & \frac{3}{2}  \left(-1+\sqrt{-11}\right) & -3 & 2 & \frac{1}{2}  \left(-1+\sqrt{-11}\right) & 4 \\
 1 & \frac{1}{2} \left(1+ \sqrt{-11}\right) & -3 & -1 & \frac{1}{2}  \left(-5 +\sqrt{-11}\right) & 6 & 1 & -\frac{3}{2} \left(1+ \sqrt{-11}\right) & -3 & 2 & -\frac{1}{2}  \left(1+\sqrt{-11}\right) & 4 \\
 1 & 1 & -2 \sqrt{2} & 2 \sqrt{2} & 1 & -2 & -1 & -2 \sqrt{2} & 2+4 \sqrt{2} & -2 \sqrt{2} & 2 \sqrt{2} & -3 \\
 1 & 1 & 2 \sqrt{2} & -2 \sqrt{2} & 1 & -2 & -1 & 2 \sqrt{2} & 2-4 \sqrt{2} & 2 \sqrt{2} & -2 \sqrt{2} & -3
\end{array}
\right)
$\\
\noalign{\hrule height0.8pt}
\end{tabular}
}
\end{center}
\end{table}
%%%%%%%%%%%%%%%%%%%%%%%%%%%%%%%%%%%%%%%%%%%%%%%%
\end{landscape}

Set $\hat{h}_i(z)\ (i=1,3,5,7,9)$ as follows 
(see Table \ref{Tab:N5} for the other cases): 
\begin{align}\label{mat:Hecke}
\begin{pmatrix}
\hat{h}_1(z)\\
\hat{h}_3(z)\\
\hat{h}_7(z)\\
\hat{h}_9(z)
\end{pmatrix}
&=
\begin{pmatrix}
 1 & -1  & -2 & -2 \\
 1 & 1  & 2 & -2 \\
 1 & 2  & -2 & 1 \\
 1 & -2  & 2 & 1 
\end{pmatrix}
\begin{pmatrix}
f_1(z)\\
f_3(z)\\
f_7(z)\\
f_9(z)
\end{pmatrix}.
\end{align}
For $i=1,3,7,9$, we denote by $c_{\hat{h}_{i}}(n)$ 
the coefficient of $\hat{h}_{i}(z)$ as follows: 
\[
\hat{h}_{i}(z)=\sum_{n=1}^{\infty}c_{\hat{h}_{i}}(n)q^n.
\]
%Let $T(n)$ be Hecke operators. 
Let $T(n)$ be the Hecke operator considered on the space 
of modular forms for $\Gamma_0(100)$ 
(see {\cite[p.~161, Proposition~37]{Kob}}). 
Then, 
%by~{\cite[Proposition 9.15]{Knapp}}
%we have 
%\begin{align*}
%T(3)f_1(z)&=3f_3(z)-2f_7(z), \\
%T(3)f_3(z)&=f_1(z), \\
%T(3)f_7(z)&=-f_9(z), \\
%T(3)f_9(z)&=f_3(z)-2f_7(z). 
%\end{align*}
% where $T(n)$ denotes the Hecke operator.
%Namely, 
$\hat{h}_i(z)\ (i=1,3,5,7,9)$ are eigen forms for 
$T(3)$. Since the algebra of Hecke operators is 
commutative~\cite[Theorem~4.5.3]{Miyake}, 
$\hat{h}_i(z)\ (i=1,3,7,9)$ are normalized Hecke eigen forms. 
In addition, for each prime $p$ and $i=1,3,7,9$, 
%the absolute value of the $p$-th coefficients 
%are bounded above by $2\sqrt{p}$. 
\[
|c_{\hat{h}_i}(p)|\leq 2\sqrt{p} 
\]
(see~\cite[p.~164]{Kob}). 

%%%%%%%%%%%%%%%%%%%%%%%%%%%%%%%%%%%%%%%%%%%%%%%%
\begin{table}[tbhp]
\caption{(\ref{eqn:bound56}) and $a(p)>0$}
\label{Tab:N6}
\begin{center}
{\small
%{\footnotesize
%{\scriptsize
\begin{tabular}{c|l|c}
\noalign{\hrule height0.8pt}
Case  & \multicolumn{1}{c|}{(\ref{eqn:bound56})} & $a(p)>0$ \\
\hline
(b) & $\frac{4}{3}\left(p+1-2\sqrt{p}\right)$ & $p > 0$ \\
\hline
(c) & $\frac{4}{21}\left(p+1-(8+4\sqrt{2})\sqrt{p}\right)$ & $p>181$ \\
\hline
(d) & $\frac{4}{15}\left(p+1-7\sqrt{p}\right)$ & $p>43$ \\
\hline
(e) & $\frac{2}{3}\left(p+1-2\sqrt{p}\right)$ & $p>0$ \\
\hline
(f) & $\frac{4}{21}\left(p+1-(2+6\sqrt{2})\sqrt{p}\right)$ & $p>107$ \\
\hline
(g) & $\frac{4}{9}\left(p+1-6\sqrt{p}\right)$ & $p>31$ \\
\hline
(h) & $\frac{2}{7}\left(p+1-4\sqrt{2}\sqrt{p}\right)$ & $p>29$ \\
\noalign{\hrule height0.8pt}
%\noalign{\hrule height0.8pt}
\end{tabular}
}
\end{center}
\end{table}
%%%%%%%%%%%%%%%%%%%%%%%%%%%%%%%%%%%%%%%%%%%%%%%%

By (\ref{mat:Hecke}), we have 
\begin{align*}
\begin{pmatrix}
f_1(z)\\
f_3(z)\\
f_7(z)\\
f_9(z)
\end{pmatrix}
=
\begin{pmatrix}
 \frac{1}{6} & \frac{1}{6} & \frac{1}{3} & \frac{1}{3}  \\
 -\frac{1}{6} & \frac{1}{6} & \frac{1}{6} & -\frac{1}{6}  \\
 -\frac{1}{6} & \frac{1}{6} & -\frac{1}{12} & \frac{1}{12}  \\
 -\frac{1}{6} & -\frac{1}{6} & \frac{1}{6} & \frac{1}{6} 
\end{pmatrix}
\begin{pmatrix}
\hat{h}_1(z)\\
\hat{h}_3(z)\\
\hat{h}_7(z)\\
\hat{h}_9(z)
\end{pmatrix}.
\end{align*}
Hence, we have
%\begin{multline*}
\begin{align*}
{f_1}(z)-11{f_3}(z)
+8{f_7}(z)-2{f_9}(z)
=
 \hat{h}_1(z)
- \frac{5}{2}\hat{h}_7(z)
+ \frac{5}{2}\hat{h}_9(z). 
\end{align*}
%\end{multline*}
For each odd prime $p$ with $p\neq 5$, 
$|c_{f_1}(p)-11c_{f_3}(p)+8c_{f_7}(p)-2c_{f_9}(p)|$ 
is bounded above by 
\[
\left(1+\frac{5}{2}+\frac{5}{2}\right) 2\sqrt{p}
=12\sqrt{p}.
\]
Using (\ref{eqn:coef56-2})
((\ref{eqn:coef56}) for the other cases), $a(p)$ is bounded below by 
\begin{align}\label{eqn:bound56}
\frac{4}{15}\left(p+1-12\sqrt{p}\right), 
\end{align}
(see Table \ref{Tab:N6} for the other cases). 
Hence, (\ref{eqn:bound56}) is positive for $p >139$, 
namely, $a(p)>0$ for $p >139$
(see Table \ref{Tab:N6} for the other cases). 
We have verified by {\sc Magma} that
%for each prime $p$ with $p<78919$ except for $p\in \{2,3,7,17,23\}$, 
$a(p)>0$
for each prime $p$ with $p \le 139$ and $p \ne 2,5,7,13,23$.
%where $a_2(p)$ is listed in Table~\ref{Tab:an2} for
%a prime $p \le 139$. 
%%%%%%%%%%%%%%%%%%%%%%%%%%%
%\begin{table}[thb]
%\caption{Coefficients $a(p)$ for primes $p \le 139$}
%\label{Tab:an2}
%\begin{center}
%{\small
%{\footnotesize
%{\scriptsize
%\begin{tabular}{cc|cc|cc|cc|cc|cc}
%\noalign{\hrule height0.8pt}
%$p$ & $a(p)$ & $p$ & $a(p)$ &  $p$ & $a(p)$ & 
%$p$ & $a(p)$ & $p$ & $a(p)$ & $p$ & $a(p)$\\
%\hline
%  2&  0& 17&  8& 41& 32& 67& 48& 97&  64& 127 & 72 \\
%  3&  0& 19&  16& 43&  32& 71& 48& 101& 56 &131& 32 \\
%  5&  8& 23&  16& 47&  32& 73& 48& 103& 72 &137& 48 \\
%  7&  8& 29&  24& 53&  40& 79& 48& 107& 80 &139& 36 \\
% 11&  8& 31& 16& 59&  40& 83& 48& 109& 80 &149& 40 \\
% 13&  8& 37& 24& 61&  48& 89& 48& 113& 88 &  &\\
%\noalign{\hrule height0.8pt}
%\end{tabular}
%}
%\end{center}
%\end{table}
%%%%%%%%%%%%%%%%%%%%%%%%%%%%%%%%%%%%%%%%%%%%%%%%
This completes the proof of Case (a). 
\end{proof}

%%%%%%%%%%%%%%%%%%%%%%%%%%%%%%%%%%
\section{Construction of $m$-frames in some unimodular lattices}
\label{sec:frame}

In this section, we provide a method for
constructing $m$-frames in unimodular lattices, which
are constructed from some self-dual $\ZZ_k$-codes by Construction A.
Combined Theorem~\ref{thm:prime} with the method, 
we construct $m$-frames in 
some extremal (optimal) odd unimodular lattices.

The following method is a generalization of 
Propositions 3.3 and 3.6 in \cite{HM12}.
Also, the cases 
$(k,m,\ell)=(4,11,2)$ and $(4,11,0)$ of the following method
can be found in \cite{Chapman} and \cite{Miezaki},
respectively.

\begin{prop}\label{prop:const}
Let $k$ be a positive integer with $k \ge 2$, and
let $\ell$ be a nonnegative integer with $\ell \le k-1$.
Let $M$ be an $n \times n$ 
%%$(0,\pm 1,\pm2,\ldots,\pm \lfloor k/2 \rfloor)$-matrix
matrix over $\ZZ$ satisfying 
\begin{equation}\label{eq:condition}
M^T=-M \text{ and } M M^T= mI_n,
\end{equation}
where $m+\ell^2 \equiv -1 \pmod{k}$.
Let $C_{2n,k}(M)$ be the self-dual $\ZZ_{k}$-code of length $2n$
with generator matrix
$\left(\begin{array}{cc}
I_n & M+ \ell I_n
\end{array}\right)$, where 
the entries of the matrix are regarded as elements of $\ZZ_{k}$.
Let $a,b,c$ and $d$ be integers with 
$b \equiv c-\ell d \pmod {k}$ and 
$d \equiv a+\ell b \pmod {k}$.
Then the set of $2n$ rows of the following matrix 
\[
F(M)=
\frac{1}{\sqrt{k}}
\left(
\begin{array}{cc}
aI_n+bM & cI_n+dM \\
-cI_n+dM & aI_n-bM
\end{array}
\right)
\]
forms a $\frac{1}{k}(a^2+m b^2+c^2+m d^2)$-frame in 
the unimodular lattice $A_{k}(C_{2n,k}(M))$.
\end{prop}
\begin{proof}
Since 
$M M^T= mI_n$ with $m+ \ell^2 \equiv -1 \pmod{k}$,
$C_{2n,k}(M)$ is a self-dual $\ZZ_{k}$-code of length $2n$.
Thus, $A_{k}(C_{2n,k}(M))$ is a unimodular lattice.
%%%
Since $C_{2n,k}(M)$ is self-dual and $M^T=-M$,
both $G=\left(\begin{array}{cc}
I_n & M+\ell I_n
\end{array}\right)$ 
and
$H=\left(\begin{array}{cc}
M-\ell I_n & I_n
\end{array}\right)$ are  generator matrices of $C_{2n,k}(M)$.

Let $s,t$ be integers.
Here, we regard the entries of the matrices $G,H$ as integers.
Then
\[
\left(\begin{array}{c}
sG+tH \\
-tG+sH
\end{array}\right)
=
\left(\begin{array}{cc}
(s-\ell t)I_n+tM & (\ell s+t)I_n+ sM \\
-(\ell s+t)I_n+sM & (s-\ell t)I_n -tM
\end{array}\right).
\]
By putting
\[
a=s-\ell t, b=t, c=\ell s+t, d=s,
\]
we have the form of $F(M)$.
Thus, 
if
$b \equiv c-\ell d \pmod {k}$ and 
$d \equiv a+\ell b \pmod {k}$ then
all rows of the matrix $F(M)$
are vectors of $A_{k}(C_{2n,k}(M))$.
Since $F(M) F(M)^T=\frac{1}{k}(a^2+mb^2+c^2+md^2)I_{2n}$,
the result follows.
\end{proof}

\begin{rem}
It follows from the assumption 
that 
$a^2+m b^2+c^2+m d^2 \equiv 0 \pmod k$.
\end{rem}

\begin{landscape}
%%%%%%%%%%%%%%%%%%%%%%%%%%%
\begin{table}[p]
\caption{Matrices $M$}
\label{Tab:D}
\begin{center}
{\small
%{\footnotesize
%{\scriptsize
\begin{tabular}{c|ccc|l|l}
\noalign{\hrule height0.8pt}
$M$ & $k$ & $m$ & $\ell$ 
 & \multicolumn{1}{c|}{$r_A$} &\multicolumn{1}{c}{$r_B$} \\
\hline
$D_{6 }$ & 3 & 25 & 1  & $(0,  2,  2)$ & $(0,  1, -4)$ \\
$P_{8 }$ & 4 &  7 & 2  && \\
%%%%
$D_{10}$ & 3 & 25 & 1  & $(0,0,2,2,0)$&$(1,2,2,-2,2)$ \\
$D'_{10}$ & 3 & 25 & 1 & $(0,0,0,0,0)$&$(-3, -2,  2, -2,  2)$ \\
$D''_{10}$ & 5 & 49 &0 & $( 0, 0,3, 3,0)$&$(-2,-3,4,-1,1)$ \\
%%%%
$D_{14}$ & 3 & 25 & 1  & $(0,2,1,0,0,1,2)$&$(-1,-2,1,-2,2,1,0)$ \\
$D'_{14}$ & 5 & 25 & 2 & $(0,0,2,-1,-1,2,0)$&$(-2,-1,-2,0,-1,-1,-2)$ \\
%% $D''_{14}$ & 9 &121& 2 & $(0,3,-3,-3,-3,-3,3)$&$(1,-4,1,-4,4,4,1)$ \\
%%%%
$D_{16}$ & 4 & 15 & 2  & $(0,1,1,0,1,0,1,1)$& $(1,1,1,-1,-1,2,-1,0)$ \\
$D_{18}$ & 6 & 49 & 2  & $(0,1,-3,0,2,2,0,-3,1)$ & $(-2,2,-1,2,1,2,1,1,1)$\\
$P_{20}$ & 4 & 19 & 0  & & \\
$D_{22}$ & 5 & 25 & 2  &  $(0,0,-1,1,0,0,0,0,1,-1,0)$
& $(1,0,-2,1,1,1,2,1,0,2,-2)$ \\
$D_{24}$ & 5 & 39 & 0  & $( 0, 1, 1, 1, 2,-1, 1,-1, 2, 1, 1, 1)$ &
$(-2,-1, 2,-1,-1,-2, 0, 1, 0, 2,-1,-1)$ \\
\noalign{\hrule height0.8pt}
\end{tabular}
}
\end{center}
\end{table}
%%%%%%%%%%%%%%%%%%%%%%%%%%%%%%%%%%%%%%%%%%%%%%%%
\end{landscape}

The matrices $P_{p+1}$ $(p=7,19)$, 
which are given in Section \ref{subsec:M},  
satisfy the assumptions in Proposition \ref{prop:const},
for the integers $k,m$ and $\ell$ listed in Table \ref{Tab:D}.
Using the form (\ref{eq:M}), we have found more 
matrices $D_n$ $(n=6,10,14,16,18,22,24)$, $D'_{n}$ $(n=10,14)$
and $D''_{10}$
satisfying the assumptions in Proposition \ref{prop:const},
where the integers $k,m$ and $\ell$  and the 
first rows $r_A$ and $r_B$ of  negacirculant matrices
$A$ and $B$ are also listed in Table \ref{Tab:D}.

By Proposition \ref{prop:const},
for matrices $M$ given in Table \ref{Tab:D},
the odd unimodular lattice
$A_k(C_{2n,k}(M))$, which is constructed from
the Type~I $\ZZ_k$-code $C_{2n,k}(M)$, contains a
$\frac{1}{k}(a^2+m b^2+c^2+m d^2)$-frame
for integers $a,b,c$ and $d$ with 
$b \equiv c-\ell d \pmod {k}$ and 
$d \equiv a+\ell b \pmod {k}$.
The minimum norms $\min(L)$ of 
the lattices $L=A_k(C_{2n,k}(M))$ listed in Table \ref{Tab:L},
which have been determined by {\sc Magma},
are also listed in the table.

% \begin{lem}\label{lem:key}
% \begin{enumerate}
% \item
% $A_3(C_{12,3}(D_6))$ contains a $k$-frame
% for a positive integer $k$
% with $k \ge 2$, $k\ne 2^{m_1} 5^{m_2} 7^{m_3}13^{m_4}23^{m_5}$, 
% \item
% $A_4(C_{16,4}(P_{8}))$  contains a $k$-frame
% for a positive integer $k$ with $k \ge 2$ and $k\ne 2^{m_1}7^{m_2}$, 
% \item
% $A_3(C_{20,3}(D_{10}))$ contains a $k$-frame
% for a positive integer $k$
% with $k \ge 2$, $k\ne 2^{m_1} 5^{m_2} 7^{m_3}13^{m_4}23^{m_5}$, 
% \item
% $A_3(C_{28,3}(D_{14}))$ contains a $k$-frame
% for a positive integer $k$
% with $k \ge 3$, $k\ne 2^{m_1} 5^{m_2} 7^{m_3}13^{m_4}23^{m_5}$, 
% \item
% $A_4(C_{32,4}(D_{16}))$ contains a $k$-frame
% for a positive integer $k$
% with $k \ge 4$, $k\ne 3^m_1$,
% \item
% $A_6(C_{36,6}(D_{18}))$ contains a $k$-frame
% for a positive integer $k$
% with $k \ge 4$, $k\ne 2^{m_1} 3^{m_2} 5^{m_3}7^{m_4}$, 
% \item
% $A_4(C_{40,4}(P_{20}))$ contains a $k$-frame
% for a positive integer $k$
% with $k \ge 4$, $k\ne 3^{m_1} 13^{m_2} 19^{m_3}$, 
% \item
% $A_5(C_{44,5}(D_{22}))$ contains a $k$-frame
% for a positive integer $k$
% with $k \ge 4$, $k\ne 2^{m_1} 3^{m_2} 17^{m_3}$, 
% \end{enumerate}
% where $m_i$ is a non-negative integer ($i=1,2,3,4,5$).
% \end{lem}

\begin{lem}\label{lem:key}
% Suppose that $L$ is any of 
% $A_k(C_{2n,k}(D_n))$ 
% $((n,k)=(6,3),(10,3),$ 
% $(14,3),(16,4),(18,6),(22,5))$,
% $A_k(C_{2n,k}(D'_n))$ $((n,k)=(10,3))$, and 
% $A_k(C_{2n,k}(P_n))$ $((n,k)=(8,4),(20,4))$ listed in 
% Table \ref{Tab:L}.
Suppose that $L$ is any of 
the lattices listed in Table \ref{Tab:L}.
Then $L$ contains a $k$-frame
for a positive integer $k$ 
satisfying the conditions ($\star$)
listed in Table \ref{Tab:L},
where $m_i$ in ($\star$) is a non-negative integer.
\end{lem}
\begin{proof}
All cases are similar, and we only give details for the
lattice $A_3(C_{12,3}(D_6))$.
Let $a,b,c$ and $d$ be integers with 
$b \equiv c-d \pmod 3$ and
$d \equiv a+b \pmod 3$.
By Proposition \ref{prop:const},
$A_3(C_{12,3}(D_6))$ contains a
$\frac{1}{3}(a^2+25b^2+c^2+25d^2)$-frame.
By Theorem \ref{thm:prime} (a),
there are integers $a,b,c$ and $d$ satisfying
$b \equiv c-d \pmod 3$,
$d \equiv a+b \pmod 3$ and 
$p=\frac{1}{3}(a^2+25b^2+c^2+25d^2)$ for each prime $p \ne 2, 5, 7, 13, 23$.
The result follows from Lemma \ref{lem:frame}.
%% The other cases requirer 
%% the cases of Theorem \ref{thm:prime}, 
%% which are listed in Table \ref{Tab:L}.
For the other lattices, Table \ref{Tab:L} lists
the cases of Theorem \ref{thm:prime}, which are used in the proof.
\end{proof}

%%%%%%%%%%%%%%%%%%%%%%%%%%%
\begin{table}[thbp]
\caption{Some unimodular lattices}
\label{Tab:L}
\begin{center}
{\small
%{\footnotesize
%{\scriptsize
\begin{tabular}{c|c|l|c}
\noalign{\hrule height0.8pt}
$L$ & $\min(L)$ & \multicolumn{1}{c|}{Conditions ($\star$)} 
&  Cases\\
%&  Theorem \ref{thm:prime}\\
\hline
$A_3(C_{12,3}(D_6))$ &2 &
$k \ge 2$, $k\ne 2^{m_1} 5^{m_2} 7^{m_3}13^{m_4}23^{m_5}$ & (a) \\
$A_4(C_{16,4}(P_{8}))$  &2 \cite{Z4-H40} &
$k \ge 2$, $k\ne 2^{m_1}7^{m_2}$ & (b)\\
$A_3(C_{20,3}(D_{10}))$ &2 &
$k \ge 2$, $k\ne 2^{m_1} 5^{m_2} 7^{m_3}13^{m_4}23^{m_5}$   & (a) \\
$A_3(C_{20,3}(D'_{10}))$ &2 &
$k \ge 2$, $k\ne 2^{m_1} 5^{m_2} 7^{m_3}13^{m_4}23^{m_5}$   & (a) \\
$A_5(C_{20,5}(D''_{10}))$ &2 &
$k \ge 2$, $k\ne 2^{m_1} 3^{m_2} 7^{m_3} 11^{m_4} 19^{m_5} 29^{m_6}$ &(c) \\
$A_3(C_{28,3}(D_{14}))$ &3 &
$k \ge 3$, $k\ne 2^{m_1} 5^{m_2} 7^{m_3}13^{m_4}23^{m_5}$ & (a)  \\
$A_5(C_{28,5}(D'_{14}))$ &3 &
$k \ge 3$, $k\ne 2^{m_1} 3^{m_2} 17^{m_3}$ & (d)\\ 
$A_4(C_{32,4}(D_{16}))$ &4 &
$k \ge 4$, $k\ne 2^{m_1} 3^{m_2}$  & (e) \\
$A_6(C_{36,6}(D_{18}))$ &4 &
$k \ge 4$, $k\ne 2^{m_1} 3^{m_2} 5^{m_3}7^{m_4}$ & (f) \\
$A_4(C_{40,4}(P_{20}))$ &4 \cite{Z4-H40} &
$k \ge 4$, $k\ne 2^{m_1}3^{m_2} 13^{m_3} 19^{m_4}$ & (g)\\
$A_5(C_{44,5}(D_{22}))$ &4 &
$k \ge 4$, $k\ne 2^{m_1} 3^{m_2} 17^{m_3}$ & (d)\\ 
%%%%
$A_5(C_{48,5}(D_{24}))$ &5 &
$k \ge 5$, $k\ne 2^{m_1} 3^{m_2} 7^{m_3} 17^{m_4}$ & (h)\\ 
%%%%%%%%%%%%%
%\hline
%$A_3(C_{20,3}(D'_{10}))$ &2 &
%$k \ge 2$, $k\ne 2^{m_1} 5^{m_2} 7^{m_3}13^{m_4}23^{m_5}$   & (a) \\
%$A_5(C_{20,5}(D''_{10}))$ &2 &
%$k \ge 2$, $k\ne 2^{m_1} 3^{m_2} 7^{m_3} 11^{m_4} 19^{m_5} 29^{m_6}$ &(c) \\
%%%%%%%%%%%%%
%$A_5(C_{28,5}(D'_{14}))$ &3 &
%$k \ge 3$, $k\ne 2^{m_1} 3^{m_2} 17^{m_3}$ & (d)\\ 
% $A_9(C_{28,9}(D''_{14}))$ &3 &
% $k \ge 3$, $k\ne 
% 2^{m_1} 3^{m_2} 5^{m_3} 7^{m_4} 11^{m_5} 13^{m_6}$  & 
% {\bf remaining} \\ 
% &&$17^{m_7} 23^{m_8} 29^{m_9} 37^{m_{10}} 41^{m_{11}} 47^{m_{12}} 
% 97^{m_{12}}$ &\\ 
\noalign{\hrule height0.8pt}
\end{tabular}
}
%{\bf (consider the order of matrices !)}
\end{center}
\end{table}
%%%%%%%%%%%%%%%%%%%%%%%%%%%%%%%%%%%%%%%%%%%%%%%%

%%%%%%%%%%%%%%%%%%%%%%%%%%%%%%%%%%%%%%%%%%%%%%%%%%%%%%%%%%%%%%%%%%%%%%%%%
\section{Frames of some extremal odd unimodular lattices
and extremal Type~I $\ZZ_k$-codes}
\label{sec:main}

In this section, we establish the existence of
a $k$-frame in some extremal (optimal) unimodular lattices 
for every positive integer $k$ with $k \ge \min(L)$.
These results yield the existence of 
an extremal Type~I $\ZZ_{k}$-code of lengths 
$n=12,16,20,32,36,40, 44$ 
and a near-extremal Type~I $\ZZ_k$-code of
length $n=28$ for a positive integer $k$,
where $k \ne 1,3$ if $n =32$ and $k \ne 1$ otherwise.

%%%%%%%%%%%%%%%%%%%%%%%%%%%%%%
\subsection{Frames of $D_{12}^+$ and Length 12}

There is a unique extremal odd unimodular lattice
in dimension $12$, up to isomorphism
(see \cite[Table 16.7]{SPLAG}),
where the lattice is denoted by $D_{12}^+$.
There is a unique binary extremal Type~I code of length $12$,
up to equivalence \cite{Pless72b},
where the code is denoted by $B_{12}$ in \cite[Table 2]{Pless72b}.
It is known that $D_{12}^+$ is constructed as $A_2(B_{12})$.
Hence, by Lemma \ref{lem:key}, we investigate the existence of
a $k$-frame in $D_{12}^+$ for $k=5,7,13,23$.

There are 16 inequivalent Type~I $\ZZ_5$-codes of 
length $12$ \cite{LPS-GF5}.
We have verified by {\sc Magma} that
the $i$th code in \cite[Table III]{LPS-GF5} gives
$D_{12}^+$ by Construction A ($i= 8, 11, 13, 16$).
There are 64 inequivalent Type~I $\ZZ_7$-codes of 
length $12$ \cite{HO02}.
We have verified by {\sc Magma} that
the code $C_{12,i}$ in \cite[Table 1]{HO02} gives
$D_{12}^+$ by Construction A
($i=11$, $12$, $15$, $17$, $20$, $38$, $42$, $43$, $47$, $49$, $51$, 
$54$, $55$, $57, \ldots, 62$).
For $k=13$ and $23$,
let $C_{k,12}$ be the $\ZZ_k$-code with generator 
matrix of the form (\ref{eq:GM}),
%% where the first rows $r_A$ and $r_B$ of $A$ and $B$
%% are listed in Table \ref{Tab:12}.
where the first rows $r_A$ and $r_B$ of $A$ and $B$
are as follows:
\[
(r_A,r_B)=
((0,  1,  6),(  2,  3,  1)) \text{ and }
((0,  1, 18),(  7,  4,  0)),
\]
respectively.
Since $AA^T+BB^T=-I_{3}$,
%% and the Euclidean weights of all rows of the 
%% generator matrix are divisible by ,
these codes are Type~I.
Moreover, we have verified 
by {\sc Magma} that 
$A_k(C_{k,12})$ ($k=13,23$) is isomorphic to $D_{12}^+$.
Hence, combined with Lemma \ref{lem:key}, we have the following:

\begin{thm}\label{thm:D12}
$D_{12}^+$ contains a $k$-frame
if and only if $k$ is a positive integer with $k \ge 2$.
% $D_{12}^+$ contains a $k$-frame
% for every positive integer $k$ with $k \ge 2$.
\end{thm}

Hence, there is a Type~I $\ZZ_k$-code $C$ with $A_k(C) \cong D_{12}^+$
for every positive integer $k$ with $k \ge 2$.
Since $D_{12}^+$ has minimum norm $2$, $C$ must be extremal.

\begin{cor}\label{cor:12}
There is an extremal Type I $\ZZ_{k}$-code of
length $12$ for every positive integer $k$ with $k \ge 2$.
\end{cor}

By Lemma \ref{lem:T}, we have the following:

\begin{cor}
Let $\sum_{m=0}^{\infty}A_{m}q^m$ denote the theta series
of $D_{12}^+$.
Then $A_{k} \ge 24$ for every positive integer $k$ with $k \ge 2$.
\end{cor}

%%%%%%%%%%%%%%%%%%%%%%%%%%%%%%%%%%%%%%%%%%%%%%%%
%% \begin{table}[thb]
%% \caption{Extremal Type I $\ZZ_{p}$-codes of length $12$ ($p=13,23$)}
%% \label{Tab:12}
%% \begin{center}
%% %{\small
%% {\footnotesize
%% %{\scriptsize
%% \begin{tabular}{c|l|l}
%% \noalign{\hrule height0.8pt}
%% $p$& \multicolumn{1}{c|}{$r_A$} & \multicolumn{1}{c}{$r_B$} \\
%% \hline
%% % 5 &$(0, 1, 3)$&$( 2, 3, 1)$ {\bf (known code!)}\\ 
%% % 7 &$(0, 1, 6)$&$( 3, 1, 1)$ {\bf (known code!)}\\
%% 13 &$(0,  1,  6)$&$(  2,  3,  1)$ \\
%% 23 &$(0,  1, 18)$&$(  7,  4,  0)$ \\
%% \noalign{\hrule height0.8pt}
%% \end{tabular}
%% }
%% \end{center}
%% \end{table}
%%%%%%%%%%%%%%%%%%%%%%%%%%%%%%%%%%%%%%%%%%%%%%%%

%%%%%%%%%%%%%%%%%%%%%%%%%%%%%%%%%
\subsection{Frames of $D_8^2$ and Length 16}

There is a unique extremal odd unimodular lattice
in dimension $16$, up to isomorphism
(see \cite[Table 16.7]{SPLAG}),
where the lattice is denoted by $D_8^2$.
There is a unique binary extremal Type~I code of length $16$,
up to equivalence \cite{Pless72b},
where the code is denoted by $F_{16}$ in \cite[Table 2]{Pless72b}.
It is known that $D_8^2$ is constructed as $A_2(F_{16})$.
Hence, by Lemma \ref{lem:key}, we investigate the existence of
a $7$-frame in $D_8^2$.

Let $C_{7,16}$ be the $\ZZ_7$-code with generator 
matrix of the form (\ref{eq:GM}),
where the first rows $r_A$ and $r_B$ of $A$ and $B$
%% are $r_A=(3, 3, 2, 1)$ and $r_B=(2, 0, 0, 0)$, respectively.
are $r_A=(0, 0, 1, 1 )$ and $r_B=(1, 3, 1, 0)$, respectively.
Since $AA^T+BB^T=-I_{4}$,
%% and the Euclidean weights of all rows of the 
%% generator matrix are divisible by ,
$C_{7,16}$ is Type~I.
We have verified by {\sc Magma} that
$A_7(C_{7,16})$ is isomorphic to $D_8^2$.
Hence, combined with Lemma \ref{lem:key}, we have the following:

\begin{thm}\label{thm:D82}
% $D_8^2$ contains a $k$-frame
% for every positive integer $k$ with $k \ge 2$.
$D_8^2$ contains a $k$-frame
if and only if $k$ is a positive integer with $k \ge 2$.
\end{thm}

\begin{cor}\label{cor:16}
There is an extremal Type I $\ZZ_{k}$-code of
length $16$ for every positive integer $k$ with $k \ge 2$.
\end{cor}

By Lemma \ref{lem:T}, we have the following:

\begin{cor}
Let $\sum_{m=0}^{\infty}A_{m}q^m$ denote the theta series
of $D_8^2$.
Then $A_{k} \ge 32$ for every positive integer $k$ with $k \ge 2$.
\end{cor}

%%%%%%%%%%%%%%%%%%%%%%%%%%%%%%%%%
\subsection{Frames of $D_4^5, A_5^4, D_{20}$ and Length 20}

There are $12$ non-isomorphic extremal odd unimodular lattices in 
dimension $20$ (see \cite[Table 2.2]{SPLAG}).
We have verified 
by {\sc Magma} that the odd unimodular lattices
$A_3(C_{20,3}(D_{10}))$, $A_3(C_{20,3}(D'_{10}))$ and
$A_5(C_{20,5}(D''_{10}))$
in Table \ref{Tab:L} are isomorphic to
the $i$th lattices $(i=11,12,1)$ in dimension $20$
in \cite[Table 16.7]{SPLAG}, where we denote the 
lattices by $D_4^5$, $A_5^4$ and $D_{20}$, respectively.
By Lemma \ref{lem:key}, we investigate the existence of
a $k$-frame in $D_4^5$ and $A_5^4$ for $k=2,5,7,13,23$, and 
a $k$-frame in $D_{20}$ for $k=2, 3,7,11,19,29$.

%%%%%%%%%%%%%%%%%%%%%%%%%%%%%%%%%%%%%%%%%%%%%%%%
\begin{table}[thb]
\caption{Extremal Type I $\ZZ_{k}$-codes of length $20$}
\label{Tab:20}
\begin{center}
%{\small
{\footnotesize
%{\scriptsize
\begin{tabular}{c|l|l|c|l|l}
\noalign{\hrule height0.8pt}
Codes& \multicolumn{1}{c|}{$r_A$} & \multicolumn{1}{c|}{$r_B$} &
Codes& \multicolumn{1}{c|}{$r_A$} & \multicolumn{1}{c}{$r_B$} \\
\hline
$C_{ 5,20}$ &$(0, 0, 0, 1, 1)$ & $( 1, 4, 2, 1, 0)$ &
$C_{ 7,20}$ &$(0, 0, 0, 1, 6)$ & $( 3, 0, 1, 1, 0)$ \\
$C_{13,20}$ &$(0,  0,  0,  1,  1)$ & $( 10,  3,  2,  1,  0)$ &
$C_{23,20}$ &$(0,  0,  0,  1, 18)$ & $(  7,  4,  0,  0,  0)$ \\
\hline
$C'_{ 5,20}$ &$(0, 0, 0, 1, 4)$&$(3, 1, 4, 1, 0)$ &
$C'_{ 7,20}$ &$(0, 0, 0, 1, 5)$&$(1, 5, 3, 1, 0)$\\
$C'_{13,20}$ &$(0, 0, 0, 1, 4)$&$(4, 0, 3, 3, 0)$ &
$C'_{23,20}$ &$(0, 0, 0, 1,12)$&$(3, 5, 7, 1, 0)$\\
\hline
$C''_{ 7,20}$& $(0, 0, 0, 1,  4)$&$(  1,  3,  2, 3, 1)$ &
$C''_{ 9,20}$& $(0,0,0,1,3)$&$(1,2,4,2,6)$ \\
$C''_{11,20}$& $(0, 0, 0, 1,  8)$&$(  5,  6,  6, 3, 2)$ &
$C''_{19,20}$& $(0, 0, 0, 1, 12)$&$( 14, 12, 11, 1, 0)$ \\
$C''_{29,20}$& $(0, 0, 0, 1, 21)$&$(  7, 11, 16, 1, 0)$ & &&\\
\noalign{\hrule height0.8pt}
\end{tabular}
}
\end{center}
\end{table}
%%%%%%%%%%%%%%%%%%%%%%%%%%%%%%%%%%%%%%%%%%%%%%%%

%%%%%%%%%%%%%%%%%  Fig  %%%%%%%%%%%%%%%%%
\begin{figure}[thb]
\centering
{\small
\begin{tabular}{ll}
$
\left(\begin{array}{cc}
A & B_1+2B_2 \\
\end{array}\right)
=
\left(\begin{array}{cc}
11& 220113303 \\
00& 021012300 \\
10& 222120030 \\
01& 031321330 \\
01& 232220201 \\
11& 231021312 \\
00& 023031002 \\
10& 230133321 \\
11& 333130022 \\
\end{array}\right),
$
&
$
2D=
\left(\begin{array}{c}
202202200 \\
220022200
\end{array}\right)
$
\end{tabular}
\caption{A generator matrix of $C'_{4,20}$}
\label{Fig:20}
}
\end{figure}
%%%%%%%%%%%%%%%%%  Fig  %%%%%%%%%%%%%%%%%

There are $7$ binary extremal Type~I code of length $20$,
up to equivalence \cite{Pless72b}.
% One of them has automorphism group of order $2^{13}\cdot3\cdot5$
% and this code is denoted by $M_{20}$ in \cite[Table 2]{Pless72b}.
The unique code with $5$ (resp.\ $45$) codewords of weight $4$ 
is denoted by $M_{20}$ (resp.\ $J_{20}$)
in \cite[Table 2]{Pless72b}.
We have verified by {\sc Magma} that
$A_2(M_{20})$ (resp.\ $A_2(J_{20})$)
is isomorphic to $D_4^5$ (resp.\ $D_{20}$).
%We have also verified by {\sc Magma} 
It is known 
that there is no binary Type~I code $C$ such that
$A_2(C)$ is isomorphic to $A_5^4$.
%%%%%
There are $6$ inequivalent ternary self-dual codes of
length $20$ and minimum weight $6$ \cite{PSW}.
We have verified by {\sc Magma} that 
$L$ is obtained from some ternary self-dual code of
length $20$ and minimum weight $6$ by Construction A
if and only if $L$ is $D_4^5$ or $A_5^4$.

Let $C_{k,20}$, $C'_{k,20}$ ($k=5,7,13,23$) and 
$C''_{k,20}$ ($k=7,9,11,19,29$)
be the $\ZZ_k$-codes with generator 
matrices of the form (\ref{eq:GM}),
where the first rows $r_A$ and $r_B$ of $A$ and $B$
are listed in Table \ref{Tab:20}.
Since $AA^T+BB^T=-I_{5}$,
these codes are Type~I.
Let $C'_{4,20}$
be the $\ZZ_4$-code with generator matrix of
the following form:
\[
\left(\begin{array}{ccc}
I_9 & A & B_1+2B_2 \\
O    &2I_2 & 2D \\
\end{array}\right),
\]
where we only list in Figure \ref{Fig:20}
the matrices
$\left(\begin{array}{cc}
A & B_1+2B_2 \\
\end{array}\right)$
and $2D$ in order to save space.
Here, $A$, $B_1$, $B_2$ and $D$ are $(1,0)$-matrices
and $O$ denotes the zero matrix.
The self-dual $\ZZ_4$-code $C'_{4,20}$ has been found by 
directly finding a $4$-frame in $A_5^4$ using {\sc Magma}.
Also, some other (new) self-dual $\ZZ_4$-codes 
are constructed in a similar way.

%%%%%%%%%%%%%%%
We have verified by {\sc Magma} that
$A_k(C_{k,20})$ is isomorphic to $D_4^5$,
$A_k(C'_{k,20})$ is isomorphic to $A_5^4$ ($k=4,5,7,13,23$), and
$A_k(C''_{k,20})$ is isomorphic to $D_{20}$ ($k=7,9,11,19,29$).
Hence, combined with Lemma \ref{lem:key}, we have the following:

\begin{thm}\label{thm:20}
$D_4^5$ contains a $k$-frame 
if and only if $k$ is a positive integer with $k \ge 2$.
$A_5^4$ contains a $k$-frame 
if and only if $k$ is a positive integer with $k \ge 3$.
$D_{20}$ contains a $k$-frame 
if and only if $k$ is a positive integer $k \ge 2$, $k \ne 3$.
% $D_4^5$ contains a $k$-frame for every positive integer $k$ with $k \ge 2$.
% $A_5^4$
% contains a $k$-frame for every positive integer $k$ with $k \ge 3$.
% $D_{20}$
% contains a $k$-frame for every positive integer $k$ with $k \ge 2$,
% $k \ne 3$.
\end{thm}

\begin{cor}\label{cor:20}
There is an extremal Type I $\ZZ_{k}$-code of
length $20$ for every positive integer $k$ with $k \ge 2$.
\end{cor}

% At dimension $20$, extremal odd unimodular lattices are divided into
% ten kinds of theta series (see \cite[Table 16.7]{SPLAG}).

By Lemma \ref{lem:T}, we have the following:

\begin{cor}
Let $\sum_{m=0}^{\infty}A_{m}q^m$ denote the theta series
of $D_4^5$ or $A_5^4$.
Then $A_{k} \ge 40$ for every positive integer $k$ with $k \ge 2$.
Let $\sum_{m=0}^{\infty}A_{m}q^m$ denote the theta series
of $D_{20}$.
Then $A_{k} \ge 40$ for every positive integer $k \ge 2$, $k \ne 3$.
\end{cor}

\begin{rem}
$D_4^5$ and $A_5^4$ have the identical
theta series 
$1 + 120 q^2 + 5120 q^3 + 67320 q^4 + 503808q^5 + \cdots$, 
and 
$D_{20}$ has theta series
$1 + 760 q^2 + 77560 q^4 + 524288 q^5 +\cdots$.
\end{rem}

%%%%%%%%%%%%%%%%%%%%%%%%%%%%%%%%%
\subsection{Length 28}

The largest minimum norm among
odd unimodular lattices in dimension $28$ is $3$.
There are $38$ non-isomorphic optimal odd unimodular
lattices in dimension $28$ \cite{BV}.
In \cite{BV}, the $38$ lattices are 
denoted by
${\mathbf{R}_{28,1}}(\emptyset)$,
${\mathbf{R}_{28,2}}(\emptyset),\ldots,
{\mathbf{R}_{28,36}}(\emptyset)$,
${\mathbf{R}_{28,37e}}(\emptyset)$,
${\mathbf{R}_{28,38e}}(\emptyset)$.
We have verified by {\sc Magma} that 
$A_3(C_{28,3}(D_{14}))$ and $A_5(C_{28,5}(D'_{14}))$ 
% $A_9(C_{28,9}(D''_{14}))$ 
in Table \ref{Tab:L}
are isomorphic to 
${\mathbf{R}_{28,32}}(\emptyset)$ and ${\mathbf{R}_{28,15}}(\emptyset)$,
% ${\mathbf{R}_{28,34}}(\emptyset)$, 
respectively.
By Lemma \ref{lem:key}, we investigate the existence of
a $k$-frame in ${\mathbf{R}_{28,32}}(\emptyset)$
for $k=4,5,7,13,23$ and 
a $k$-frame in ${\mathbf{R}_{28,15}}(\emptyset)$
for $k=3,4,17$. % , and
% a $k$-frame in ${\mathbf{R}_{28,34}}(\emptyset)$
% for $k=3,4,5,7,11,13,17,23,29,37$, $41,47,97$.

%%%%%%%%%%%%%%%%%%%%%%%%%%%%%%%%%%%%%%%%%%%%%%%%
\begin{table}[thb]
\caption{Near-extremal Type I $\ZZ_{k}$-codes of length $28$}
\label{Tab:28}
\begin{center}
%{\small
{\footnotesize
%{\scriptsize
\begin{tabular}{c|l|l}
\noalign{\hrule height0.8pt}
Codes & \multicolumn{1}{c|}{$r_A$} & \multicolumn{1}{c}{$r_B$} \\
\hline
$C_{ 5,28}$ &$(0,0,0,1,3,4,2)$ & $(3,1,2,0,3,4,0)$\\
$C_{ 7,28}$ &$(0,1,2,2,4,2,3)$ & $(2,2,4,0,4,1,2)$\\
$C_{13,28}$ &$(0, 0, 0, 1, 0, 9, 1)$ & $(5, 1, 3, 7, 7, 1, 4)$\\
$C_{23,28}$ &$(0, 0, 0, 1,12, 1, 1)$ & $(3,19, 7, 5,14,21,17)$\\
\hline
$C'_{17,28}$ & $( 0,0,0,1,13,14, 2)$&  $(10,1,1,9,16,11,15)$ \\
% 5 &$(0, 0, 0, 1, 4, 4, 3)$& $(2, 0, 3, 3, 2, 0, 1)$\\
% 7 &$(0, 0, 0, 1, 3, 6, 2)$& $(6, 5, 6, 0, 2, 6, 1)$\\
%13 &$(0, 0, 0, 1, 5, 4, 3)$& $(6, 3, 0, 9, 9, 0, 1)$\\
%23 &$(0, 0, 0, 1, 0,21, 7)$& $(3, 5,19,17,19, 0, 2)$\\
% \hline
% &{\bf remaining case} & \\
% $C''_{5 ,28}$& $(0, 0, 0, 1, 3, 2, 0)$&$(  2, 1, 1, 2, 4, 2, 0)$\\
% $C''_{7 ,28}$& $(0, 0, 0, 1, 5, 5, 6)$&$(  2, 5, 6, 3, 2, 6, 1)$\\
% $C''_{11,28}$& $(0, 0, 0, 1, 1, 1, 5)$&$(  4, 2, 3, 7, 8, 9,10)$\\
% $C''_{13,28}$& $(0, 0, 0, 1,10, 7, 7)$&$( 12, 3, 5, 4, 9,10, 7)$\\
% $C''_{17,28}$& $(0, 0, 0, 1, 6, 2, 4)$&$(  4, 6, 5, 1, 0, 4, 1)$\\
% $C''_{23,28}$& $(0, 0, 0, 1, 5,20,20)$&$(  8,10,19, 6,22,10,12)$\\
% $C''_{29,28}$& $(0, 0, 0, 1,23,27, 3)$&$( 22,19,16,24, 3, 6,24)$\\
% $C''_{37,28}$& $(0, 0, 0, 1,29,17, 3)$&$( 33,30, 8,33,36, 6,34)$\\
% $C''_{41,28}$& $(0, 0, 0, 1,34,17,24)$&$( 35,15,39, 4,33,21,26)$\\
% $C''_{47,28}$& $(0, 0, 0, 1,36,14, 5)$&$(  6, 9,45, 1,14,32,37)$\\
% $C''_{97,28}$& $(1,67,26,59,28,91,76)$&$( 63,45,64,79,94,43,56)$\\
\noalign{\hrule height0.8pt}
\end{tabular}
}
\end{center}
\end{table}
%%%%%%%%%%%%%%%%%%%%%%%%%%%%%%%%%%%%%%%%%%%%%%%%

Let $C_{k,28}$ ($k=5,7,13,23$) and $C'_{17,28}$
% $C''_{k,28}$ ($k=5,7,11,13,17,23,29$, $37$, $41$, $47$, $97$)
be the $\ZZ_k$-codes with generator 
matrices of the form (\ref{eq:GM}),
where the first rows $r_A$ and $r_B$ of $A$ and $B$
are listed in Table \ref{Tab:28}.
Since $AA^T+BB^T=-I_{7}$,
these codes are Type~I.
Let 
$C_{4,28}$ and $C'_{4,28}$  % and $C''_{4,28}$
be the $\ZZ_4$-codes with generator matrices of
the following form:
\[
\left(\begin{array}{ccc}
I_{13} & A & B_1+2B_2 \\
O    &2I_{2} & 2D \\
\end{array}\right),
\]
where we list in Figure \ref{Fig}
the matrices
$
\left(\begin{array}{cc}
A & B_1+2B_2 \\
2I_{2} & 2D \\
\end{array}\right).
$
Then these codes are Type~I.
For $k=4,5,7,13,23$,
we have verified by {\sc Magma} that
$A_k(C_{k,28})$ is isomorphic to ${\mathbf{R}_{28,32}}(\emptyset)$.
%%%%
For $k=4,17$,
we have verified by {\sc Magma} that
$A_k(C'_{k,28})$ is isomorphic to ${\mathbf{R}_{28,15}}(\emptyset)$.
% For $k=4,5,7,11,13,17,23,29,37,41$, $47$, $97$, 
% we have verified by {\sc Magma} that
% $A_k(C''_{k,28})$ is isomorphic to 
% ${\mathbf{R}_{28,34}}(\emptyset)$ ({\bf remaining case}).
It is known that
${\mathbf{R}_{28,15}}(\emptyset)$ contains a $3$-frame
% ${\mathbf{R}_{28,34}}(\emptyset)$ contains no $3$-frame 
(see  \cite{HMV} for the classification of $3$-frames in 
the $38$ lattices).
%%%%
Hence, combined with Lemma \ref{lem:key}, we have the following:

\begin{thm}\label{thm:28}
${\mathbf{R}_{28,i}}(\emptyset)$ $(i=15,32)$
contains a $k$-frame 
if and only if $k$ is a positive integer with $k \ge 3$.
%${\mathbf{R}_{28,15}}(\emptyset)$  contains a $k$-frame 
%if and only if $k$ is a positive integer with $k \ge 3$.
% The following are equivalent:
% \begin{enumerate}
% \item
% ${\mathbf{R}_{28,32}}(\emptyset)$ contains a $k$-frame.
% \item
% ${\mathbf{R}_{28,15}}(\emptyset)$  contains a $k$-frame.
% \item
% $k$ is a positive integer with $k \ge 3$.
% \end{enumerate}
%%%%%%%%%%%%%
% ${\mathbf{R}_{28,34}}(\emptyset)$ contains a $k$-frame 
% if and only if $k$ is a positive integer with $k \ge 4$
% ({\bf remaining case}).
% ${\mathbf{R}_{28,32}}(\emptyset)$ and ${\mathbf{R}_{28,15}}(\emptyset)$ 
% contain a $k$-frame 
% for every positive integer $k$ with $k \ge 3$.
% ${\mathbf{R}_{28,34}}(\emptyset)$ contains a $k$-frame 
% for every positive integer $k$ with $k \ge 4$.
\end{thm}

\begin{lem}
Let $C$ be a Type~I $\ZZ_k$-code of length $28$.
Then $d_E(C) \le 3k$ 
for every positive integer $k$ with $k \ge 2$.
\end{lem}
\begin{proof}
As described above,
the largest minimum norm among
odd unimodular lattices in dimension $28$ is $3$.
Assume that $k \ge 4$ and $d_E(C) =4k$.
Since $\min(A_{k}(C))=\min\{k, d_{E}(C)/k\}$,
$\min(A_{k}(C)) =4$, which is a contradiction.
For $k=2,3$, it follows that $d_E(C) \le 3k$ (see \cite{CPS, MPS}).
\end{proof}

By the above lemma, there is no extremal Type~I $\ZZ_k$-code
of length $28$ for every positive integer $k$ with $k \ge 2$.
There is a binary Type I code of length $28$
and minimum weight $6$ (see \cite{CPS}).
Hence, we have the following:

\begin{cor}\label{cor:28}
There is no extremal Type I $\ZZ_{k}$-code of
length $28$ for every positive integer $k$ with $k \ge 2$.
There is  a near-extremal Type I $\ZZ_{k}$-code of
length $28$ for every positive integer $k$ with $k \ge 2$.
\end{cor}

Since ${\mathbf{R}_{28,i}}(\emptyset)$ $(i=1,2,\ldots,36)$
have the identical theta series,
by Lemma \ref{lem:T}, we have the following:

\begin{cor}
Let $\sum_{m=0}^{\infty}A_{m}q^m$ denote the theta series
of ${\mathbf{R}_{28,i}}(\emptyset)$ $(i=1,2,\ldots,36)$.
Then $A_{k} \ge 56$ for every positive integer 
$k$ with $k \ge 3$.
\end{cor}

%%%%%%%%%%%%%%%%%  Fig  %%%%%%%%%%%%%%%%%
\begin{figure}[thb]
\centering
%{\small
{\footnotesize
$
\left(\begin{array}{cc}
00&3221032113010\\
00&2312302202000\\
01&1011113132031\\
01&2021011201031\\
10&3033332032202\\
00&2220031132311\\
00&1130232110223\\
11&2213122020013\\
01&3200201111201\\
01&3133230220230\\
10&3111000202123\\
10&3011332120200\\
10&3331011112112 \\
20&2220022000000 \\
02&0022222000000
\end{array}\right) \text{ and }
\left(\begin{array}{cc}
01& 1023301203302\\
01& 1022200000021\\
01& 1130203022312\\
11& 1202303012212\\
00& 2321232113032\\
01& 0002112332213\\
11& 1113323310300\\
00& 3321000023111\\
00& 1210231221321\\
11& 2012013002211\\
11& 1010001123020\\
11& 2203101320001\\
00& 3302011030033\\
20& 0002202020200 \\
02& 2220222202200
\end{array}\right) % ,
% \left(\begin{array}{cc}
% 01& 0202013303211\\
% 10& 1131223313132\\
% 00& 2300332001331\\
% 10& 1120323222310\\
% 10& 1220231112232\\
% 01& 0230123000111\\
% 00& 0030112112213\\
% 01& 2221100031233\\
% 01& 3313200233022\\
% 10& 0002313122303\\
% 00& 1233101232030\\
% 00& 2020101311312\\
% 11& 2301101321333\\
% 20& 2002222000000\\
% 02& 2220202000000
% \end{array}\right)
$
\caption{Generator matrices of $C_{4,28}$ and $C'_{4,28}$}
\label{Fig}
}
\end{figure}
%%%%%%%%%%%%%%%%%  Fig  %%%%%%%%%%%%%%%%%

%%%%%%%%%%%%%%%%%%%%%%%%%%%%%%%%%
\subsection{Length 32}

% \begin{rem}
% We have verified by {\sc Magma} that 
% $C_{4,32}$ has weight enumerator $1+8y^4+700y^8+\cdots$,
% and $A_4(C_{4,32})$
% has theta series
% $1 + 81344 q^4 + 2097152 q^5 + 32251904 q^6 +\cdots$.
% \end{rem}

There are $5$ non-isomorphic extremal odd unimodular lattices in 
dimension $32$, and 
these $5$ lattices are related to the $5$ inequivalent 
binary extremal Type~II codes of length $32$ \cite{CS-odd},
where the $5$ codes are denoted by 
$\text{C}81,\text{C}82,\ldots,\text{C}85$ 
in \cite[Table A]{CPS}.
We denote the extremal odd unimodular lattice
related to $\text{C}i$ by $L_{32,i}$ ($i=81,\ldots,85$).
We have verified 
by {\sc Magma} that the odd unimodular lattice
$A_4(C_{32,4}(D_{16}))$  in Table \ref{Tab:L}
is isomorphic to $L_{32,82}$.
Since $A_4(C_{32,4}(D_{16}))$ contains a $4$-frame,
we investigate the existence of
a $k$-frame in $L_{32,82}$ for $k=6,9$ by Lemma \ref{lem:key}.

%%%%%%%%%%%%%%%%%%%%%%%%%%%%%%%%%%%%%%%%%%%%%%%%
\begin{table}[thb]
\caption{Extremal Type I $\ZZ_{k}$-codes of length $32$}
\label{Tab:32}
\begin{center}
%{\small
{\footnotesize
%{\scriptsize
\begin{tabular}{c|l|l}
\noalign{\hrule height0.8pt}
Codes & \multicolumn{1}{c|}{$r_A$} & \multicolumn{1}{c}{$r_B$} \\
\hline
$C_{6,32}$ & $(0,0,1,2,2,2,1,2)$ & $(1,0,5,5,1,1,3,3)$ \\
$C_{9,32}$ & $(0,0,1,5,0,6,0,1)$ & $(0,6,2,2,7,6,1,7)$ \\
\noalign{\hrule height0.8pt}
\end{tabular}
}
\end{center}
\end{table}
%%%%%%%%%%%%%%%%%%%%%%%%%%%%%%%%%%%%%%%%%%%%%%%%

For $k=6,9$, let $C_{k,32}$ be the
$\ZZ_k$-code  with generator matrix of the form (\ref{eq:GM}),
where the first rows $r_A$ and $r_B$ of $A$ and $B$
are listed in Table \ref{Tab:32}.
% Let $C_{9,32}$ be the $\ZZ_9$-code with generator 
% matrix of the form (\ref{eq:GM}),
% where the first rows $r_A$ and $r_B$ of $A$ and $B$
% are as follows:
% \[
% r_A=(0,0,0,3,4,3,0,0) \text{ and }
% r_B=(1,8,2,5,4,0,4,1),
% \]
% respectively.
Since $AA^T+BB^T=-I_{8}$,
these codes are Type~I.
For $k=6,9$, we have verified by {\sc Magma} that
$A_k(C_{k,32})$ is isomorphic to $L_{32,82}$.
Hence, combined with Lemma \ref{lem:key}, we have the following:

\begin{thm}\label{thm:32}
$L_{32,82}$ contains a $k$-frame
if and only if $k$ is a positive integer with $k \ge 4$.
% $L_{32,82}$ contains a $k$-frame
% for every positive integer $k$ with $k \ge 4$.
\end{thm}

There are three inequivalent
binary extremal Type I codes of length $32$ \cite{C-S}.
Any ternary self-dual code of length $32$ has minimum weight 
at most $9$ \cite{MS73}.
% There is an extremal Type~I $\ZZ_4$-code of 
% length $32$ \cite{H-FFA12}.
Hence, we have the following:

\begin{cor}\label{cor:32}
There is an extremal Type I $\ZZ_{k}$-code of
length $32$ for every positive integer $k$ with $k \ne 1,3$.
\end{cor}

Since the $5$ non-isomorphic extremal odd unimodular lattices 
have the identical theta series  \cite{CS-odd}, 
by Lemma \ref{lem:T}, we have the following:

\begin{cor}
Let $\sum_{m=0}^{\infty}A_{m}q^m$ denote the theta series
of an extremal odd unimodular lattice in dimension $32$.
Then $A_{k} \ge 64$ for every positive integer 
$k$ with $k \ge 4$.
\end{cor}

%%%%%%%%%%%%%%%%%
For each extremal odd unimodular lattice in dimension $32$,
one of the even unimodular neighbors is 
extremal \cite{CS-odd}.
Moreover, it follows from the construction in \cite{CS-odd}
that the extremal even unimodular neighbor of $L_{32,82}$ 
is the $32$-dimensional Barnes--Wall lattice $BW_{32}$
(see e.g. \cite[Chapter 8, Section 8]{SPLAG} for $BW_{32}$).
Since the even sublattice of $L_{32,82}$ contains
a $2k$-frame for every positive integer $k$ with $k \ge 2$
by Theorem \ref{thm:32},
we have the following:

\begin{prop}
$BW_{32}$ contains a $2k$-frame
if and only if $k$ is a positive integer with $k \ge 2$.
\end{prop}

% Hence, the extremal even neighbor of $L_{32,82}$ 
% contains a $2k$-frame
% for a positive integer $k$.
% The existence of an extremal Type~II $\ZZ_{2k}$-code of 
% length $32$ is known for $k=1,2,\ldots,6$ (see \cite[Table 1]{HM}).
% The existence of an extremal Type~II $\ZZ_{6}$-code of 
% length $32$ is known (see \cite{HKO}).
Then we have an alternative proof of the following:

\begin{cor}[Harada and Miezaki \cite{HM12}]
There is an extremal Type II $\ZZ_{2k}$-code of
length $32$ for every positive integer $k$.
\end{cor}

%%%%%%%%%%%%%%%%%%%%%%%%%%%%%%%%%
\subsection{Length 36}

Since $A_6(C_{36,6}(D_{18}))$ contains a $6$-frame,
we investigate the existence of
a $k$-frame in $A_6(C_{36,6}(D_{18}))$ for $k=4,5,7,9$
by Lemma \ref{lem:key}.
For $k=5,7,9$, let $C_{k,36}$ be the
$\ZZ_k$-code  with generator matrix of the form (\ref{eq:GM}),
where the first rows $r_A$ and $r_B$ of $A$ and $B$
are listed in Table \ref{Tab:36}.
Since $AA^T+BB^T=-I_{9}$,
these codes are Type~I.
% An extremal Type~I $\ZZ_{4}$-code of length $36$ can be
% found in \cite[Table 1]{H12}.
Let $C_{4,36}$ be the $\ZZ_4$-code with generator matrix of
the following form:
\[
\left(\begin{array}{ccc}
I_{16} & A & B_1+2B_2 \\
O    &2I_{4} & 2D \\
\end{array}\right),
\]
where we only list in Figure \ref{Fig:36}
the matrices
$\left(\begin{array}{cc}
A & B_1+2B_2 \\
\end{array}\right)$
and $2D$.
For $k=4,5,7,9$, 
we have verified by {\sc Magma} that
$A_k(C_{k,36})$ is isomorphic to $A_6(C_{36,6}(D_{18}))$.
Hence, combined with Lemma \ref{lem:key}, we have the following:

\begin{thm}\label{thm:36}
$A_6(C_{36,6}(D_{18}))$ contains a $k$-frame 
if and only if $k$ is a positive integer with $k \ge 4$.
% $A_6(C_{36,6}(D_{18}))$ contains a $k$-frame 
% for every positive integer $k$ with $k \ge 4$.
\end{thm}
\begin{rem}
We have verified by {\sc Magma} that $A_6(C_{36,6}(D_{18}))$
has theta series $1 + 42840q^4 + 1916928q^5 + 42286080q^6 +
\cdots$ and automorphism group of order $288$.
\end{rem}

%%%%%%%%%%%%%%%%%%%%%%%%%%%%%%%%%%%%%%%%%%%%%%%%
\begin{table}[thb]
\caption{Extremal Type I $\ZZ_{k}$-codes of length $36$}
\label{Tab:36}
\begin{center}
%{\small
{\footnotesize
%{\scriptsize
\begin{tabular}{c|l|l}
\noalign{\hrule height0.8pt}
Codes & \multicolumn{1}{c|}{$r_A$} & \multicolumn{1}{c}{$r_B$} \\
\hline
$C_{5,36}$ & $(0,1,1,2,3,2,0,2,3)$ & $(1,1,0,2,0,3,4,0,4)$ \\
$C_{7,36}$ & $(0,1,6,2,3,3,6,4,5)$ & $(4,3,3,6,2,4,3,0,3)$ \\
$C_{9,36}$ & $(0,1,0,5,5,0,0,0,3)$ & $(0,2,3,3,4,5,5,7,3)$ \\
% 5&$(0,0,0,1,2,2,0,4,2)$&$(0,0,0,0,4,3,3,0,1)$\\
% 7&$(0,0,0,1,5,5,5,1,1)$&$(0,5,4,2,5,1,1,3,6)$\\
% 9&$(0,0,0,1,0,4,3,0,0)$&$(0,4,1,5,3,5,1,7,0)$\\
%11&$( 0, 0, 0, 1, 6, 0, 3, 0, 6)$&$( 5,10, 5, 0, 1, 9, 5, 6, 8)$\\
%19&$( 0, 0, 0, 1,15,15, 9, 6, 5)$&$(14,16, 0,14,15, 8, 8, 3,12)$\\
%29&$( 0, 0, 0, 1,25, 5,17, 9,10)$&$(14, 8, 8, 3,19,27, 6, 4,10)$\\
\noalign{\hrule height0.8pt}
\end{tabular}
}
\end{center}
\end{table}
%%%%%%%%%%%%%%%%%%%%%%%%%%%%%%%%%%%%%%%%%%%%%%%%

%%%%%%%%%%%%%%%%%  Fig  %%%%%%%%%%%%%%%%%
\begin{figure}[thb]
\centering
%{\small
{\footnotesize
\begin{tabular}{ll}
$
\left(\begin{array}{cc}
A & B_1+2B_2 \\
\end{array}\right)
=
\left(\begin{array}{cc}
0100 & 1203131221301121 \\
1011 & 1011202100200000 \\
1010 & 2020221222311322 \\
0101 & 1311223022101123 \\
1110 & 0022223222133220 \\
0110 & 0102101300313130 \\
0100 & 1003131232103103 \\
0001 & 0212210231101002 \\
1101 & 3311103322131110 \\
0101 & 3033123233020103 \\
0101 & 1320133200323130 \\
0100 & 2002221022321133 \\
0101 & 3211333002312322 \\
0101 & 1031113220233320 \\
0101 & 0103111200301112 \\
1110 & 2222020200331300 
\end{array}\right),
$
&
$
2D=
\left(\begin{array}{c}
0020020000222022\\
0200202000000220\\
2222220000020000\\
2022000000000000
\end{array}\right)
$
\end{tabular}
\caption{A generator matrix of $C_{4,36}$}
\label{Fig:36}
}
\end{figure}
%%%%%%%%%%%%%%%%%  Fig  %%%%%%%%%%%%%%%%%

There are $41$ inequivalent binary extremal Type I codes of 
length $36$ \cite{MG08}.
There is a ternary extremal Type I code of length $36$ \cite{Pless72}.
Hence, we have the following:

\begin{cor}\label{cor:36}
There is an extremal Type I $\ZZ_{k}$-code of
length $36$ for every positive integer $k$ with $k \ge 2$.
\end{cor}

By Lemma \ref{lem:T}, we have the following:

\begin{cor}
Let $\sum_{m=0}^{\infty}A_{m}q^m$ denote the theta series
of $A_6(C_{36,6}(D_{18}))$.
Then $A_{k} \ge 72$ for every positive integer $k$ with $k \ge 4$.
\end{cor}

%%%%%%%%%%%%%%%%%%%%%%%%%%%%%%%%%%%%%%%%%%%%%%%%%
\subsection{Length 40}

Since $A_4(C_{40,4}(P_{20}))$ contains a $4$-frame,
we investigate the existence of
a $k$-frame in extremal odd unimodular lattices in 
dimension $40$ for $k=6,9,13,19$ by Lemma \ref{lem:key}.
For $k=9,13,19$,
let $C_{k,40}$ be the $\ZZ_k$-code with generator 
matrix of the form (\ref{eq:GM}),
where the first rows $r_A$ and $r_B$ of $A$ and $B$
are listed in Table \ref{Tab:40}.
Since $AA^T+BB^T=-I_{10}$, these codes are Type~I.
Moreover, we have verified 
by {\sc Magma} that $A_k(C_{k,40})$ is extremal ($k=9,13,19$).
An extremal Type~I $\ZZ_{6}$-code of length $40$ can be
found in \cite{GH05}.
Hence, combined with Lemma \ref{lem:key}, we have the following:

\begin{lem}\label{lem:40}
There is an extremal odd unimodular lattice in dimension $40$
containing a $k$-frame if and only if
$k$ is a  positive integer $k$ with $k \ge 4$.
\end{lem}

\begin{rem}
The possible theta series of an extremal odd unimodular
lattice in dimension $40$ is given in \cite{BBH}:
$\theta_{40,\alpha}(q)=1 + (19120  + 256 \alpha) q^4
+ (1376256 - 4096 \alpha) q^5 + \cdots$,
where $\alpha$ is even with $0 \le \alpha \le 80$.
We have verified by {\sc Magma} that $A_4(C_{40,4}(P_{20}))$ 
has theta series $\theta_{40,80}(q)$
and automorphism group of order $7172259840$, and 
$A_k(C_{k,40})$ ($k=9,13,19$) 
have theta series $\theta_{40,0}(q)$
and automorphism group of order $40$.
% We have verified by {\sc Magma} that $A_4(C_{40,4}(P_{20}))$ 
% has theta series 
% $1 + 39600 q^4 + 1048576 q^5 + 45916160 q^6 + \cdots$
% and automorphism group of order $7172259840$, and 
% $A_k(C_{k,40})$ ($k=9,13,19$) 
% have theta series 
% $1 + 19120 q^4 + 1376256 q^5 + 43950080 q^6 +\cdots$
% and automorphism group of order $40$.
Also, we have verified by {\sc Magma} that
three lattices $A_k(C_{k,40})$ ($k=9,13,19$) 
are non-isomorphic.
\end{rem}

For $k=2,3$, 
there is an extremal Type I $\ZZ_k$-code of length $40$.
Hence, we have the following:

\begin{thm}\label{thm:40}
There is an extremal Type I $\ZZ_{k}$-code of
length $40$ for every positive integer $k$ with $k \ge 2$.
\end{thm}

%%%%%%%%%%%%%%%%%%%%%%%%%%%%%%%%%%%%%%%%%%%%%%%%
\begin{table}[thb]
\caption{Extremal Type I $\ZZ_{k}$-codes of length $40$}
\label{Tab:40}
\begin{center}
%{\small
{\footnotesize
%{\scriptsize
\begin{tabular}{c|l|l}
\noalign{\hrule height0.8pt}
Codes & \multicolumn{1}{c|}{$r_A$} & \multicolumn{1}{c}{$r_B$} \\
\hline
$C_{ 9,40}$ &$(0,0,1,0,5,8,3,0,4,4)$ & $(0,5,0,0,5,6,7,2,5,8)$ \\
$C_{13,40}$ &$(0,0,1,4,10, 5, 1,10,11, 4)$ & $(11,4, 4,6, 7,12,11, 7, 2,8)$\\
$C_{19,40}$ &$(0,0,1,2,14,16,17, 1, 0,13)$ & $(10,2,15,2,18,16, 9,15,12,0)$\\
%9 & $(1,8,1,0,0,0,0,0,0,0)$ & $(0,0,0,1,1,0,7,5,1,6)$\\
%13& $( 8, 0, 9, 0, 3,11, 3, 2, 0, 9)$
%                  &$(6, 5, 3, 2, 7, 1, 2,10, 5, 1)$\\ 
%19& $( 9, 6,16, 1,15, 3, 1,10,12,14)$
%                  &$( 1, 2,10,17, 5, 9, 5, 9, 9, 7)$\\
\noalign{\hrule height0.8pt}
\end{tabular}
}
\end{center}
\end{table}
%%%%%%%%%%%%%%%%%%%%%%%%%%%%%%%%%%%%%%%%%%%%%%%%

We have verified by {\sc Magma} that
at least one of the even unimodular neighbors of
$L$ is extremal 
for $L=A_4(C_{40,4}(P_{20}))$, $A_9(C_{9,40})$,
$A_{13}(C_{13,40})$ and $A_{19}(C_{19,40})$.
% The existence of an extremal Type~II $\ZZ_{2k}$-code of 
% length $40$ is known for $k=1,2,\ldots,6$ (see \cite[Table 1]{HM}).
There are binary extremal Type~II codes of length $40$
(see \cite{BHM} for their classification).
Then we have an alternative proof of the following:

\begin{prop}[Harada and Miezaki \cite{HM12}]
There is an extremal Type II $\ZZ_{2k}$-code of
length $40$ for every positive integer $k$.
\end{prop}

%%%%%%%%%%%%%%%%%%%%%%%%%%%%%%%%%
\subsection{Length 44}

By Lemma \ref{lem:key}, we investigate the existence of
a $k$-frame in extremal odd unimodular lattices in 
dimension $44$ for $k=4,6,9,17$.
For $k=9,17$,
let $C_{k,44}$ be the $\ZZ_k$-code with generator 
matrix of the form (\ref{eq:GM}),
where the first rows $r_A$ and $r_B$ of $A$ and $B$
are listed in Table \ref{Tab:44}.
Since $AA^T+BB^T=-I_{11}$, 
these codes are Type~I.
Moreover, we have verified 
by {\sc Magma} that $A_k(C_{k,44})$ is extremal ($k=9,17$).
For $k=4,6$, 
an extremal Type~I $\ZZ_{k}$-code of length $44$ can be
found in \cite[Table 1]{H12} and \cite{GH05}, respectively.
Hence, combined with Lemma \ref{lem:key}, we have the following:

\begin{lem}\label{lem:44}
There is an extremal odd unimodular lattice in dimension $44$
containing a $k$-frame if and only if 
$k$ is a  positive integer $k$ with $k \ge 4$.
\end{lem}

\begin{rem}
The possible theta series of an extremal odd unimodular
lattice in dimension $44$ is given in \cite{H03}:
$\theta_{44,1,\beta}(q)=1+(6600+16\beta)q^4+(811008-128\beta)q^5 + \cdots$,
$\theta_{44,2,\beta}(q)=1+(6600+16\beta)q^4+(679936-128\beta)q^5 + \cdots$,
where $\beta$ is an integer.
We have verified by {\sc Magma} that 
$A_5(C_{44,5}(D_{22}))$ and 
$A_k(C_{k,44})$ ($k=9,17$)
have theta series 
$\theta_{44,1,\beta}(q)$ ($\beta=0,88,176$), 
and automorphism groups of orders
$44$, $88$ and $44$, respectively.
% We have verified by {\sc Magma} that 
% $A_5(C_{44,5}(D_{22}))$ and 
% $A_k(C_{k,44})$ ($k=9,17$)
% have theta series 
% $1 + 6600 q^4 + 811008 q^5 + 37171200 q^6 + \cdots$,
% $1 + 8008 q^4 + 799744 q^5 + 37159936 q^6 + \cdots$ and
% $1 + 9416 q^4 + 788480 q^5 + 37148672 q^6 + \cdots$,
% and automorphism groups of orders
% $44$, $88$ and $44$, respectively.
\end{rem}
%%%%%%%%%%%%%%%%%%%%%%%%%%%%%%%%%%%%%%%%%%%%%%%%
\begin{table}[thb]
\caption{Extremal Type I $\ZZ_{k}$-codes of length $44$}
\label{Tab:44}
\begin{center}
%{\small
{\footnotesize
%{\scriptsize
\begin{tabular}{c|l|l}
\noalign{\hrule height0.8pt}
Codes & \multicolumn{1}{c|}{$r_A$} & \multicolumn{1}{c}{$r_B$} \\
\hline
$C_{9 ,44}$& $(0,0,0,0,1,0,1,4,0,8,0)$ & $(7,0,7,1,8,8,2,8,1,5,1)$ \\
$C_{17,44}$& $(0, 0, 0, 0, 1,13, 7,13,11,16,13)$ 
    &$(12,14, 8,14, 7,12,14, 7,14,14, 7)$ \\
\noalign{\hrule height0.8pt}
\end{tabular}
}
\end{center}
\end{table}
%%%%%%%%%%%%%%%%%%%%%%%%%%%%%%%%%%%%%%%%%%%%%%%%

For $k=2,3$, 
there is an extremal Type I $\ZZ_k$-code of length $44$.
Hence, we have the following:

\begin{thm}\label{thm:44}
There is an extremal Type I $\ZZ_{k}$-code of
length $44$ for every positive integer $k$ with $k \ge 2$.
\end{thm}

%%%%%%%%%%%%%%%%%%%%%%%%%%%%%%
\section{Remarks}

We end this paper with some remarks about
the existence of
a $k$-frame in optimal odd unimodular lattices in 
dimension $48$.

By Lemma \ref{lem:key}, we investigate the existence of
a $k$-frame in optimal odd unimodular lattices in 
dimension $48$ for $k=6, 7, 8, 9, 17$.  It was shown in \cite{HKMV} that
an extremal even unimodular lattice in dimension $48$
has an optimal odd unimodular neighbor.
Using this result, we have the following:

\begin{lem}\label{lem:6-1}
There is an optimal odd unimodular lattice in dimension $48$
containing an $8k$-frame for every positive integer $k$.
\end{lem}
\begin{proof}
Let $\Lambda$ be 
an extremal even unimodular lattice in dimension $48$.
Let $x$ be a vector of $\Lambda$ with $(x,x)=8$.
Put $\Lambda_x^{+}=\{v\in\Lambda\mid(x,v)\equiv0\pmod2\}$.
Since there is a vector $y$ of $\Lambda$ such
that $(x,y)$ is odd, the following lattice
\[
\Lambda_{x}=
\Lambda_x^+ \cup \Big(\frac{1}{2}x+y\Big)+\Lambda_x^+
\]
is an optimal odd unimodular neighbor of $\Lambda$ \cite{HKMV}.

Some extremal even unimodular lattice in
dimension $48$ containing an $8$-frame can be found 
in \cite[Corollary 1]{Chapman-Sole}.
We take this lattice as $\Lambda$ in the above construction.
Let $\{f_1, \ldots, f_{48}\}$ be an $8$-frame in $\Lambda$.
Then $\Lambda_{f_1}$ is an optimal odd unimodular neighbor
containing $\{f_1, \ldots, f_{48}\}$.
The result follows from Lemma \ref{lem:frame}.
\end{proof}

%%%%%%%%%%%%%%%%%%%%%%%%%%%%%%%%%%%%%%%%%%%%%%%%
\begin{table}[thb]
\caption{Near-extremal Type I $\ZZ_{k}$-codes of length $48$}
\label{Tab:48}
\begin{center}
%{\small
{\footnotesize
%{\scriptsize
\begin{tabular}{c|l|l}
\noalign{\hrule height0.8pt}
Codes & \multicolumn{1}{c|}{$r_A$} & \multicolumn{1}{c}{$r_B$} \\
\hline
$C_{7 ,48}$& $(0,1,6,3,0,2,0,2,4,2,5,3)$&$(3,6,1,5,4,6,0,5,0,5,1,5)$\\
$C_{9 ,48}$& $(0,1,2,4,6,1,6,2,2,0,3,0)$&$(7,2,5,1,6,8,4,1,2,2,8,4)$\\
\noalign{\hrule height0.8pt}
\end{tabular}
}
\end{center}
\end{table}
%%%%%%%%%%%%%%%%%%%%%%%%%%%%%%%%%%%%%%%%%%%%%%%%

Some near-extremal Type~I $\ZZ_6$-code $C_{6,48}$ of length $48$
can be found in \cite{HKMV}.
% there is an optimal odd unimodular lattice in dimension $48$
% containing a $6k$-frame for every positive integer $k$.
For $k=7,9$,
let $C_{k,48}$ be the $\ZZ_k$-code with generator 
matrix of the form (\ref{eq:GM}),
where the first rows $r_A$ and $r_B$ of $A$ and $B$
are listed in Table \ref{Tab:48}.
Since $AA^T+BB^T=-I_{12}$, 
these codes are Type~I.
Moreover, we have verified 
by {\sc Magma} that $A_k(C_{k,48})$ is optimal ($k=7,9$).
Hence, we have the following:

\begin{lem}\label{lem:48}
There is an optimal odd unimodular lattice in dimension $48$
containing a $k$-frame 
for every positive integer $k$ with $k \ge 5$ and
$k \ne 2^{m_1}3^{m_2}17^{m_3}$,
where 
$m_i$ are non-negative integers $(i=1,2,3)$ with 
$(m_1,m_2) \in \{(0,0),(0,1),(1,0),(2,0)\}$ and $m_3 \ge 1$.
% $(m_1, m_2) \in \{(0,a) \mid a \in \ZZ, a \ge 0\} \cup
% \{(1,0),(2,0)\}$.
\end{lem}

\begin{rem}
$A_6(C_{6,48})$ has kissing number $393216$
\cite[p.~553]{HKMV}.
In addition, we have verified by {\sc Magma} that
$A_5(C_{48,5}(D_{24}))$,
$A_7(C_{7,48})$ and
$A_9(C_{9,48})$ have
kissing number $393216$.
\end{rem}

%arXiv% \begin{rem}
%arXiv% $A_6(C_{6,48})$ has an extremal even unimodular 
%arXiv% neighbor \cite[Proposition 4.4]{HKMV} and 
%arXiv% $\Lambda_x$ in the proof of Lemma \ref{lem:6-1}
%arXiv% has an extremal even unimodular neighbor.
%arXiv% We have verified by {\sc Magma} that
%arXiv% $A_5(C_{48,5}(D_{24}))$ has an extremal even unimodular neighbor,
%arXiv% $A_7(C_{7,48})$ has an extremal even unimodular neighbor, and
%arXiv% $A_9(C_{9,48})$ has an extremal even unimodular 
%arXiv% neighbor {\bf (now checking ($k=9$) ???)}.
%arXiv% \end{rem}
%arXiv% 

There are at least $264$ inequivalent binary near-extremal Type I 
code of length $48$ \cite{BYK}.
There are at least two inequivalent ternary near-extremal Type I code 
of length $48$ \cite{Pless72}.
It is not known whether there is a near-extremal Type I 
$\ZZ_4$-code of length $48$ (see \cite{H12}).

\begin{prop}
There is a near-extremal Type I $\ZZ_{k}$-code of length $48$
for integers $k=2,3$ and 
for integers $k$ with $k \ge 5$,
$k \ne 2^{m_1}3^{m_2}17^{m_3}$,
where 
$m_i$ are non-negative integers $(i=1,2,3)$ with 
$(m_1,m_2) \in \{(0,0),(0,1),(1,0),(2,0)\}$ and $m_3 \ge 1$.
% for integers $k$ with $k \ge 5$,
% $k \ne 2^{m_1}3^{m_2}17^{m_4}$,
% where $m_i$ are non-negative integers $(i=1,2,3)$ with 
% $(m_1, m_2) \in \{(0,a) \mid a \in \ZZ, a \ge 0\} \cup\{(1,0),(2,0)\}$.
\end{prop}

%arXiv% Using the method which is the same as that in the previous
%arXiv% section, we tried to construct a
%arXiv% near-extremal Type~I $\ZZ_{17}$-code of length $48$.
%arXiv% However, our extensive search failed to discover such
%arXiv% a code then we stopped our search at length $48$.
%arXiv% It is worthwhile to determine whether there is a
%arXiv% near-extremal Type~I $\ZZ_k$-code of length $48$ $(k=4,17)$.

%%%%%%%%%%%%%%%%%  References  %%%%%%%%%%%%%%%%%%%%%%%%


\begin{thebibliography}{99}

\bibitem{BV} R. Bacher and B. Venkov, 
R\'eseaux entiers unimodulaires sans racines en dimensions 27 et 28,
R\'eseaux euclidiens, designs sph\'eriques et formes modulaires, 
212--267, 
Monogr.\ Enseign.\ Math., 37, Enseignement Math., Geneva, 2001.

% \bibitem{BB09}E. Bannai and Ets. Bannai,
% A survey on spherical designs and algebraic combinatorics on spheres,
% {\sl European J. Combin.}
% {\bf 30} (2009),  1392--1425.

\bibitem{BDHO} E.~Bannai, S.T.~Dougherty, M.~Harada and M.~Oura,
{Type~II codes, even unimodular lattices and invariant rings,}
{\sl IEEE\ Trans.\ Inform.\ Theory}
{\bf 45} (1999), 257--269.

% \bibitem{BM} E.\ Bannai and T.\ Miezaki,
% {Toy models for D.\ H.\ Lehmer's conjecture},
% {\sl J. Math.\ Soc.\ Japan} 
% {\bf 62} (2010), 687--705.

\bibitem{BHM} K. Betsumiya, M. Harada and A. Munemasa,
A complete classification of doubly even self-dual codes of length 40,
{\sl Electronic J. Combin.}
{\bf 19} (2012), \# P18 (12 pp.).

% \bibitem{Z4-BSBM} A. Bonnecaze, P. Sol\'e, C. Bachoc and B. Mourrain,
% {Type~II codes over $\ZZ_4$},
% {\sl IEEE\ Trans.\ Inform.\ Theory}
% {\bf 43} (1997), 969--976.

\bibitem{Magma}W. Bosma, J.J. Cannon, C. Fieker and A. Steel,
{\sl Handbook of Magma Functions (Edition 2.18)}, 2011.

\bibitem{BBH} S. Bouyuklieva, I. Bouyukliev and M. Harada,
Some extremal self-dual codes and unimodular lattices in dimension $40$,
(submitted), arXiv:1111.2637.

\bibitem{BYK}S. Bouyuklieva, N. Yankov and J.-L. Kim,
Classification of binary self-dual $[48, 24, 10]$ codes with
an automorphism of odd prime order,
{\sl Finite Fields Appl.}
{\bf  18} (2012), 1104--1113.

\bibitem{Chapman}R. Chapman,
Double circulant constructions of the Leech lattice,
{\sl J. Austral.\ Math.\ Soc.\ Ser.~A}
{\bf 69} (2000), 287--297.

\bibitem{Chapman-Sole} R.~Chapman and P.~Sol\'e,
{Universal codes and unimodular lattices},
{\sl J. The\'or.\ Nombres Bordeaux}
{\bf 8} (1996), 369--376.

\bibitem{CPS} J.H. Conway, V. Pless and N.J.A. Sloane,
The binary self-dual codes of length up to $32$: a revised enumeration,
{\sl J.\ Combin. Theory Ser.~A},
{\bf 60} (1992), 183--195.

% \bibitem{Z4-C-S} J.H.\ Conway and N.J.A.\ Sloane,
% {Self-dual codes over the integers modulo 4,}
% {\sl J.\ Combin.\ Theory\ Ser.~A}
% {\bf 62} (1993),  30--45.


\bibitem{C-S} J.H.~Conway and N.J.A.~Sloane,
A new upper bound on the minimal distance of self-dual codes,
{\sl IEEE\ Trans.\ Inform.\ Theory}
{\bf 36} (1990), 1319--1333.

\bibitem{CS-odd}J.H.~Conway and N.J.A.~Sloane,
{A note on optimal unimodular lattices},
{\sl J.\ Number Theory}
{\bf 72} (1998), 357--362.

\bibitem{SPLAG} J.H. Conway and N.J.A. Sloane,
{\sl Sphere Packing, Lattices and Groups (3rd ed.)},
Springer-Verlag, New York, 1999.

% \bibitem{DHS} S.T. Dougherty, M. Harada and P. Sol\'e,
% Self-dual codes over rings and the Chinese remainder theorem,
% {\sl Hokkaido Math.\ J.}
% {\bf 28} (1999), 253--283.

\bibitem{Gaulter} M. Gaulter,
{Minima of odd unimodular lattices in dimension $24m$},
{\sl J. Number Theory}
{\bf 91} (2001), 81--91.

% \bibitem{GS}A.V. Geramita and J. Seberry, 
% {\sl Orthogonal Designs, Quadratic Forms
% and Hadamard Matrices}, Lecture Notes in Pure and Applied Mathematics,
% 45, Marcel Dekker Inc., New York, 1979.

\bibitem{GH01} T.A. Gulliver and M. Harada,
{Orthogonal frames in the Leech lattice and a Type~II code over $\ZZ_{22}$}, 
{\sl J.\ Combin.\ Theory Ser.\ A}
{\bf 95} (2001), 185--188.

\bibitem{GH05}T.A. Gulliver and M. Harada, 
Extremal self-dual codes over $\Bbb Z_6,\ \Bbb Z_8$ and $\Bbb Z_{10}$,
{\sl AKCE Int.\ J.\ Graphs Comb.}
{\bf 2} (2005),  11--24.

% \bibitem{Z4-H}M. Harada,
% New extremal Type~II codes over $\ZZ_4$,
% {\sl Des.\ Codes Cryptogr.}
% {\bf  13} (1998),  271--284. 

\bibitem{Z4-H40}M. Harada,
Self-dual $\ZZ_4$-codes and Hadamard matrices,
{\sl Discrete Math.}
{\bf 245} (2002), 273--278.

\bibitem{H03}M. Harada,
Extremal odd unimodular lattices in dimensions 44, 46 and 47, 
{\sl Hokkaido Math.\ J.}
{\bf 32} (2003) 153--159.

\bibitem{H12}M. Harada,
Optimal self-dual $\ZZ_4$-codes and a unimodular lattice in dimension 41, 
{\sl Finite Fields Appl.}
{\bf 18} (2012), 529--536.

\bibitem{HKMV} M. Harada, M. Kitazume, A. Munemasa and B. Venkov,
{On some self-dual codes and unimodular lattices in dimension 48},
{\sl European J. Combin.}
{\bf 26} (2005), 543--557.
 
% \bibitem{HKO}M. Harada, M. Kitazume and M. Ozeki,
% Ternary code construction of unimodular lattices and 
% self-dual codes over $\Bbb Z_6$,
% {\sl J. Algebraic Combin.}
% {\bf 16}  (2002), 209--223.

% \bibitem{HM}M. Harada and T. Miezaki,
% An upper bound on the minimum weight of Type~II $\Bbb Z_{2k}$-codes,
% {\sl J. Combin.\ Theory Ser.~A}
% {\bf 118} (2011), 190--196.

\bibitem{HM12}M. Harada and T. Miezaki,
On the existence of extremal Type II $\Bbb Z_{2k}$-codes,
{\sl Math.\ Comp.}, (to appear),
arXiv: 11205.6947.

\bibitem{HMV} M. Harada, A. Munemasa and B. Venkov,
{Classification of ternary extremal self-dual codes 
of length $28$}, 
{\sl Math.\ Comp.} 
{\bf 78} (2009), 1787--1796. 

\bibitem{HO02} M. Harada and P.R.J. \"Osterg\aa rd,
Self-dual and maximal self-orthogonal codes over ${\Bbb F}_7$,
{\sl Discrete Math.}
{\bf 256}  (2002),  471--477.

% \bibitem{HO03} M. Harada and P.R.J. \"Osterg\aa rd,
%  On the classification of self-dual codes over~$\FF_5$,
% {Graphs Combin.}
% {\bf 19} (2003), 203--214. 

% \bibitem{Z4-HSG} M. Harada, P. Sol\'e and P. Gaborit,
% {Self-dual codes over $\ZZ\sb 4$ and unimodular lattices: a survey},
% Algebras and combinatorics (Hong Kong, 1997), 
% Springer, Singapore, 1999, pp.~255--275.

% \bibitem{Huffman05}W.C. Huffman, 
% On the classification and enumeration of self-dual codes,
% {\sl Finite Fields Appl.}
% {\bf  11}  (2005),  451--490.

% \bibitem{JR}P. Jenkins and J. Rouse,
% Bounds for coefficients of cusp forms and extremal lattices,
% {\sl Bull.\ London Math.\ Soc.}
% {\bf 43} (2011), 927--938.

% \bibitem{Knapp} A. Knapp, 
% {\sl Elliptic Curves},
% Princeton Univ.\ Press, Princeton, 1992. 

\bibitem{Kob} N. Koblitz,
{\sl Introduction to Elliptic Curves and Modular Forms},
Springer-Verlag, Berlin/New York, 1984. 

\bibitem{LPS-GF5} J.S. Leon, V. Pless and N.J.A. Sloane,
{Self-dual codes over {\rm GF$(5)$}},
{\sl J.\ Combin.\ Theory~Ser.~A}
{\bf 32} (1982), 178--194.

\bibitem{MOS75}C.L. Mallows, A.M. Odlyzko and N.J.A. Sloane,
{Upper bounds for modular forms, lattices, and codes},
{\sl J. Algebra}
{\bf 36} (1975), 68--76.

\bibitem{MPS}C.L. Mallows, V. Pless and N.J.A. Sloane,
{Self-dual codes over $GF(3)$},
{\sl SIAM J. Appl.\ Math.}
{\bf 31} (1976), 649--666.

\bibitem{MS73}C.L. Mallows and N.J.A. Sloane, 
An upper bound for self-dual codes, 
{\sl Inform.\ Control}
{\bf  22} (1973), 188--200. 

\bibitem{MG08}C.A. Melchor and P. Gaborit,
On the classification of extremal binary self-dual codes,
{\sl IEEE\ Trans.\ Inform.\ Theory} 
{\bf 54} (2008),  4743--4750.

\bibitem{Miezaki}T. Miezaki,
Frames in the odd Leech lattice,
{\sl J.\ Number Theory}
{\bf 132} (2012), 2773--2778.

\bibitem{Miyake} T. Miyake, 
{\sl Modular Forms}, 
Translated from the Japanese by Y. Maeda, %Yoshitaka Maeda. 
Springer-Verlag, Berlin, 1989.

\bibitem{Murty}M.R. Murty, 
{Congruences between modular forms}, 
{Analytic number theory (Kyoto, 1996)}, 
London Math.\ Soc.\ Lecture Note Ser., 
247, Cambridge Univ.\ Press, Cambridge, 1997, pp.~309--320.

% \bibitem{Nebe72}G. Nebe,
% An even unimodular $72$-dimensional lattice of minimum 8,
% {\sl J. Reine Angew.\ Math.},
% (to appear), arXiv: 1008.2862.

% \bibitem{S-http} G.~Nebe and N.J.A.~Sloane,
% {Unimodular lattices, together with a table of the best such lattices},
% in {\sl A Catalogue of Lattices},
% published electronically at 
% {\small \verb+http://www.math.rwth-aachen.de/~Gabriele.Nebe/LATTICES/+}

% \bibitem{Pache}C. Pache, 
% Shells of selfdual lattices viewed as spherical designs,
% {\sl Internat.\ J. Algebra Comput.}
% {\bf 15}  (2005),  1085--1127.

\bibitem{Pless72} V. Pless,
{Symmetry codes over $GF(3)$ and new five-designs},
{\sl J.\ Combin.\ Theory Ser.~A}
{\bf 12} (1972), 119--142.

\bibitem{Pless72b}V. Pless,
A classification of self-orthogonal codes over ${\rm GF}(2)$,
{\sl Discrete Math.}
{\bf  3}  (1972), 209--246.

\bibitem{PSW} V. Pless, N.J.A. Sloane and H.N. Ward,
{Ternary codes of minimum weight 6 and the classification of
length 20},
{\sl IEEE\ Trans.\ Inform.\ Theory}
{\bf 26} (1980), 305--316.

\bibitem{Rains} E.M.~Rains,
Shadow bounds for self-dual codes,
{\sl IEEE\ Trans.\ Inform.\ Theory}
{\bf 44} (1998), 134--139.

\bibitem{RS-bound} E.~Rains and N.J.A.~Sloane,
{The shadow theory of modular and unimodular lattices},
{\sl J.\ Number Theory}
{\bf 73} (1998), 359--389.

\bibitem{RS-Handbook} E.\ Rains and N.J.A.\ Sloane,
{Self-dual codes,} {Handbook of Coding Theory},
V.S. Pless and W.C. Huffman (Editors),
Elsevier, Amsterdam, 1998, pp.\ 177--294.

% \bibitem{Shimakua} H. Shimakura, 
% {Decompositions of the Moonshine module with respect to 
% subVOAs associated to codes over $\ZZ_{2k}$},
% {\sl J. Algebra} {\bf 251}  (2002),  308--322.

\bibitem{Siegel} C.L. Siegel, 
{Berechnung von Zetafunktionen an ganzzahligen Stellen} (German), 
{\sl Nachr.\ Akad.\ Wiss.\ G\"ottingen Math.-Phys. Kl.\ II} {\bf 1969} 
(1969), 87--102. 

% \bibitem{Venkov} B.B. Venkov, 
% {Even unimodular extremal lattices}, (Russian) 
% {\sl Algebraic geometry and its applications. Trudy Mat.\ Inst.\ Steklov.}
% {\bf 165} (1984), 43--48; 
% translation in {\sl Proc. Steklov Inst.\ Math.} {\bf 165} (1985), 
% 47--52. 

\end{thebibliography}
\end{document}